\documentclass[11pt]{amsart}
\usepackage{multirow,bigdelim}
\usepackage{amsfonts}
\usepackage{dsfont}
\usepackage{amsmath}
\usepackage{mathrsfs}
\usepackage{todonotes}
\usepackage{hyperref}
\usepackage{multimedia}
\usepackage{graphicx}
\usepackage{indentfirst}
\usepackage{bm}
\usepackage{amsthm}
\usepackage{mathrsfs}
\usepackage{latexsym}
\usepackage{amsmath}
\usepackage{amssymb}
\usepackage{microtype}
\usepackage{relsize}
\usepackage{hyperref}

\DeclareSymbolFont{SY}{U}{psy}{m}{n}
\DeclareMathSymbol{\emptyset}{\mathord}{SY}{'306}

\theoremstyle{plain}

\newtheorem{thm}{Theorem}[section]
\newtheorem{cor}[thm]{Corollary}
\newtheorem{lem}[thm]{Lemma}
\newtheorem{prop}[thm]{Proposition}
\newtheorem{defn}[thm]{Definition}

\theoremstyle{definition}
\newtheorem{rem}[thm]{Remark}

\numberwithin{equation}{section}

\def\beq{\begin{eqnarray}}
\def\eeq{\end{eqnarray}}
\def\beqa{\begin{eqnarray*}}
\def\eeqa{\end{eqnarray*}}

\begin{document}
\title[A note on unitary equivalence of operators acting on reproducing kernel Hilbert spaces]{A note on unitary equivalence of operators acting on reproducing kernel Hilbert spaces }

\author{Kui Ji and Shanshan Ji$^{*}$}
\curraddr[K. Ji and S. Ji]{Department of Mathematics Science, Hebei Normal University, Shijiazhuang, Hebei 050016, China}
\email[K. Ji]{jikuikui@gmail.com, jikui@hebtu.edu.cn}
\email[S. Ji]{jishanshan15@outlook.com}

\subjclass[2000]{Primary 47C15, 47B37; Secondary 47B48, 47L40}
\keywords{Cowen-Douglas operator, Reproducing kernel, Unitarily equivalent, (IR) decomposition}
\thanks{* Corresponding author}

\begin{abstract}
A well-known theorem due to R. E. Curto and N. Salinas gives a necessary and sufficient condition for the unitary equivalence of commuting tuples of bounded linear operators acting on reproducing kernel Hilbert spaces \cite{CS2}.
Inspired by this theorem, we obtain a different but equivalent criterion for the unitary equivalence of operators acting on reproducing kernel Hilbert spaces.
As an application, we describe the structure of intertwining operator and  prove that the decomposition of Cowen-Douglas operators is unique up to unitary equivalence.
\end{abstract}
\maketitle
\section{Introduction}
Let $\mathcal H$ be an infinite dimensional complex and separable Hilbert space, ${\mathcal L}({\mathcal H})$ the
collection of all bounded linear operators on $\mathcal H$.
Let $T, \widetilde{T}\in{\mathcal L}({\mathcal H})$. We say $T$ is unitarily equivalent to $\widetilde{T}$ (denoted by $T \sim_{u} \widetilde{T}$), if there is a unitary operator $U\in{\mathcal L}({\mathcal H})$ such that $UT = \widetilde{T}U$.
In the paper \cite{CD1}, M. J. Cowen and R. G. Douglas introduced the class $\mathcal{B}_{n}(\Omega)$ of bounded linear operators on ${\mathcal H}$.
They proved that every operator $T\in\mathcal{B}_{n}(\Omega)$ determines a Hermitian holomorphic vector bundle $E_{T}$ with the property:
two Cowen-Douglas operators are unitarily equivalent if and only if the corresponding Hermitian holomorphic vector bundles are congruent.
They also found a set of complete invariants for congruence consisting of the curvature of $E_{T}$ and its covariant partial derivatives. Unfortunately, these invariants are not easy
to compute except when the rank of the bundle is one.
K. Zhu gives another equivalent condition for unitary equivalence of Cowen-Douglas operators by using spanning holomorphic cross-section in \cite{ZKH}.
In \cite{AM}, A. Koranyi and G. Misra classify homogeneous operators in $\mathcal{B}_{n}(\Omega)$ modulo unitary equivalence.

Let $T\in\mathcal{B}_{n}(\Omega)$. In \cite{CD1}, it is proved that
$T$ is unitarily equivalent to the adjoint of multiplication operator $M_{z}$ on some reproducing kernel
Hilbert space consisting of holomorphic $\mathbb{C}^{n}-$valued functions defined on $\Omega^{*}:=\{\lambda: \overline{\lambda}\in \Omega\}$.
Without loss of generality, $T\in\mathcal{B}_{n}(\Omega)$ can be viewed as $M^{*}_{z}$, that is, $T=(M^{*}_{z}, {\mathcal H}, K)$,
where the Hilbert space $\mathcal{H}$ consisting of holomorphic functions possessing a reproducing kernel $K$ taking values in $n\times n$ complex matrices.
Furthermore, R. E. Curto and N. Salinas also provide a necessary and sufficient condition for unitary equivalence of
two multiplication operators $M_{z}$ acting on two different Hilbert spaces $\mathcal{H}_{1}, \mathcal{H}_{2}$ with reproducing kernels $K_{1}, K_{2}$\ (see Remark 3.8, \cite{CS2}).

\begin{thm}\cite{CS2}\label{L1}
Let $T_{i}\in \mathcal{B}_{n}(\Omega)$, and $T_{i}=(M^{*}_{z}, \mathcal{H}_{i}, K_{i}),i=1,2$,
where $\mathcal{H}_{i}$ are the analytic function spaces with reproducing kernels $K_{1},K_{2}$, respectively.
Then $T_{1}\sim_{u}T_{2}$ if and only if
\begin{equation}\label{(4.8990)}
K_{1}(\mu, \lambda)=\Phi(\mu)K_{2}(\mu, \lambda)\Phi(\lambda)^{*}
\end{equation}
for a $\mathcal{M}_{n}(\mathbb{C})-$valued holomorphic function $\Phi$ such that $\Phi(\mu)$ is invertible for every $\mu$, where $\mu, \lambda$ in some open subset $\Lambda$ of\, $\Omega^{*}$.
\end{thm}

In \cite{C12}, it is shown that an intertwiner $X$ between two operators of the form $(M^{*}_{z}, \mathcal{H}_{1}, K_{1})$ and $(M^{*}_{z}, \mathcal{H}_{2}, K_{2})$ exist if and only if there is a fixed linear map $A$ such that $AK_1(z,w)=K_2(z,w)A$ assuming that $K_1$ and $K_2$ are normalized kernels.
Now we obtain a useful alternative criterion for unitary equivalence similar to Theorem \ref{L1} (see Theorem \ref{KK1}).

As is well-known, in finite dimensional space, every $n\times n$ matrix can be uniquely written as an orthogonal direct sum of irreducible matrices up to unitary equivalence.
Compared with this classic conclusion, it is natural to ask if we can arrive at a similar conclusion for irreducible decomposition of operators acting on an infinite dimensional Hilbert space.
In \cite{CFJ4}, Y. Cao, J. Fang and C. Jiang show operators with unique finite (IR) decomposition up to unitary equivalence by using $K_{0}-$group of their commutants.
As an application of the unitary equivalence criterion of R. E. Curto and N. Salinas given above, we prove the uniqueness of the decomposition of reducible Cowen-Douglas operators up to unitary equivalence.

The paper is organized as follows. In Section 2, we introduce some notations and definitions about Cowen-Douglas theory.
In Section 3, motivated by Theorem \ref{L1} due to R. E. Curto and N. Salinas, there is an equivalent condition for unitary equivalence of operators acting on reproducing kernel Hilbert spaces in a different form.
In Section 4, 
we prove that the (IR) decomposition of Cowen-Douglas operators is unique up to unitary equivalence. In Section 5, by using Theorem \ref{KK1} and Theorem \ref{Thm}, we describe the structure of the intertwining matrix between the two reproducing kernels.

\section{Preliminaries}
In this section, we review some basic facts and notions for Cowen-Douglas operators and reproducing kernel (see \cite{CD1}, \cite{CS2}).
\subsection{Cowen-Douglas operators}
\begin{defn}\cite{CD1}\label{CD def}
Let $\Omega$ be an open connected subset of complex plane $\mathbb{C}$ and $n$ be a positive integer. The class $\mathcal{B}_{n}(\Omega)$ consists of operators $T\in{\mathcal L}({\mathcal H})$ satisfying:
\begin{itemize}
  \item [(1)]$\Omega\subset\sigma(T)=\{\lambda\in \mathbb{C}:T-\lambda$ is not invertible$\}$;
  \item [(2)]ran$(T-\lambda)=\mathcal{H}$ for every $\lambda$ in $\Omega$;
  \item [(3)]$\bigvee _{\lambda{\in}{\Omega}} \ker(T-\lambda)=\mathcal H$;
  \item [(4)]$\dim \ker (T-\lambda)=n$ for every $\lambda$ in $\Omega$.
\end{itemize}
\end{defn}
It is shown in the same paper that
each operator $T\in\mathcal{B}_{n}(\Omega)$ determines a Hermitian holomorphic vector bundle $(E_{T}, \pi, \lambda)$ over $\Omega$,
namely, the sub-bundle of the trivial bundle $\Omega\times\mathcal H$ defined by
$$E_{T}:=\{(\lambda,x)\in\Omega\times\mathcal H|\ x\in \ker(T-\lambda), \lambda\in\Omega\} $$
with the natural projection map $\pi:E_{T}\rightarrow \Omega$, $\pi(\lambda, x)=\lambda$.
Since dim\ ker$(T-\lambda)=n$ for every $\lambda\in\Omega$, $E_{T}$ is of rank $n$.
Let $\{\gamma_{i}\}_{i=1}^{n}$ be the holomorphic frame of $E_{T}$ (Proposition 1.12, \cite{CD1}), then form the matric at $\lambda$ of inner products
$h(\lambda):=(\langle \gamma_{j}(\lambda),\gamma_{i}(\lambda)\rangle)_{n\times n}, \lambda\in\Omega.$

Let $\xi(\Omega)$ be the algebra consist of the $C^{\infty}$ functions and $\xi^{p}(\Omega)$ denote the $p-$defferential form of $C^{\infty}$ functions. Thus, we have
$$\xi^{0}(\Omega)=\xi(\Omega),\ \xi^{1}(\Omega)=\{fdz+gd\overline{z}: f,g\in \xi(\Omega)\},\ \xi^{2}(\Omega)=\{fdzd\overline{z}: f\in \xi(\Omega)\}.$$
For any vector bundle $E$ which has $C^{\infty}$ differential structure, let $\xi^{p}(\Omega, E)$ denote $p-$differential forms with the coefficients in $E$. Then each element in $\xi^{0}(\Omega, E)$ is one of sections of $E$.
The connection $D$ can be regarded as a differential operator which maps $\xi^{0}(\Omega, E)$ to $\xi^{1}(\Omega, E)$.
Then the canonical connection $D$ preserves the metric and satisfies the following:
$$D(\sum\limits_{i=1}^{n}f_{i}\gamma_{i})=\sum\limits_{i=1}^{n}df_{i}\otimes\gamma_{i}+\sum\limits_{i=1}^{n}\sum\limits_{j=1}^{n}f_{i}\Theta_{j,i}\gamma_{i},$$
with the connection matrix $\Theta=(\Theta_{ij})_{n\times n}=h^{-1}\frac{\partial}{\partial\lambda} h$. And
$$D^{2}=d\Theta+\Theta\wedge\Theta=\frac{\partial}{\partial\overline{\lambda}}(h^{-1}\frac{\partial}{\partial\lambda} h),$$ then $-\frac{\partial}{\partial\overline{\lambda}}(h^{-1}\frac{\partial}{\partial\lambda} h)$ is called the curvature of $E$, denoted by $\mathcal{K}_E$.
In fact, the curvature of $E_{T}$ depends on the choice of the frame, and the curvature under different choices is similar except when $n=1$. Based on a simple calculation, when $T\in\mathcal{B}_{1}(\Omega)$, we have
$$\mathcal{K}_{T}(\lambda)=-\frac{\partial^{2}}{\partial\lambda\partial\overline{\lambda}}\log\parallel e(\lambda)\parallel^{2},$$
where $e(\lambda)\in \ker(T-\lambda)$ is a non-zero cross section of $E_{T}$.
In this case, the curvature does not depend on the choice of the non-vanishing section $e$. If $\widetilde{e}$ is another
non-vanishing holomorphic section, then $\widetilde{e}=\phi e$, where $\phi$ is a holomorphic function
on $\Omega_{0}\subset\Omega$, and $\log|\phi|$ is harmonicity, so curvature is completely unitary invariant.
However, if $T\in\mathcal{B}_{n}(\Omega)(n>1)$, only the eigenvalues of the curvature
are independent of the choice of the holomorphic frame.

The curvature can be regarded as a bundle map, and by the definition of bundle map on a holomorphic vector bundle, we obtain the covariant partial derivatives
$\mathcal{K}_{T,\omega^{i}\overline{\omega}^{j}}$ of the curvature $\mathcal{K}_{T}$:
\begin{itemize}
\item[(1)] $\mathcal{K}_{T,\lambda^{i}\overline{\lambda}^{j+1}}=\frac{\partial}{\partial\overline{\lambda}}(\mathcal{K}_{T,\lambda^{i}\overline{\lambda}^{j}})$;
\item[(2)] $\mathcal{K}_{T,\lambda^{i+1}\overline{\lambda}^{j}}=\frac{\partial}{\partial\lambda}(\mathcal{K}_{T,\lambda^{i}\overline{\lambda}^{j}})+[h^{-1}\frac{\partial}{\partial\lambda}h,\mathcal{K}_{T,\lambda^{i}\overline{\lambda}^{j}}]$.
\end{itemize}
Here is one of the main results from \cite{CD1}.
\begin{thm}\cite{CD1}\label{L1.2}
Let $T,S\in\mathcal{B}_{n}(\Omega)$. Then $T\sim_{u}S$ if and only if there exists an isometric bundle map $V:E_{T}\rightarrow E_{S}$ such that
$$V(\mathcal{K}_{T,\lambda^{i}\overline{\lambda}^{j}})=(\mathcal{K}_{S,\lambda^{i}\overline{\lambda}^{j}})V,\ 0\leq i, j\leq i+j\leq n,\ (i,j)\neq(0,n), (n,0).$$
\end{thm}
Thus the curvature $\mathcal{K}_{T}$ and its covariant partial derivatives
$\mathcal{K}_{T, \lambda^{i}\overline{\lambda}^{j}}$ of holomorphic bundle $E_{T}$ corresponding to $T\in \mathcal{B}_{n}(\Omega)$ are complete unitary invariants.
In \cite{MS}, G. Misra and S. Shyam Roy show that the class of the curvature alone does not necessarily determine the class of the vector bundle except in the case of a line bundle.
\subsection{Reproducing kernel}
We use symbol $\mathcal{M}_{n}(\mathbb{C})$ to denote the set of all $n\times n$ matrices over $\mathbb{C}$.
Here are some properties of the reproducing kernel (Lemma 3.2, \cite{CS2}).
Given a bounded domain $\Lambda$,
let $K: \Lambda\times\Lambda\rightarrow \mathcal{M}_{n}(\mathbb{C})$ be a non-negative kernel, holomorphic in the first variable
and anti-holomorphic in the second variable.
As is well-known, there is a $\mathbb{C}^n$-value Hilbert space $\mathcal{H}$ contained
in the space of holomorphic functions defined on $\Lambda$ such that $K(\cdot, \lambda)\xi$ is in $\mathcal{H}$ for all $\lambda\in\Lambda$ and $\xi\in \mathbb{C}^n$, $K$
serves as the reproducing kernel for $\mathcal{H}$, that is, $\langle f, K(\cdot, \lambda)\rangle=f(\lambda)$ for all $f\in\mathcal{H}$. The
Hilbert space $\mathcal{H}$ is said to be analytic function space with reproducing kernel $K$.
Clearly, $K(\mu, \lambda)=K(\lambda, \mu)^{*}$, where $K(\lambda, \mu)^{*}$ is conjugate transpose of $K(\lambda, \mu),\ \mu, \lambda\in\Lambda$.


In \cite{CD1}, M. J. Cowen and R. G. Douglas described the relationship between $T\in\mathcal{B}_{1}(\Omega)$ and multiplication operator on some Hilbert space ${\mathcal H}$.
Let $\gamma$ be a non-vanishing holomorphic cross-section of $E_{T}$ (which exists by Grauert's theorem \cite{HG}). Naturally, corresponding to $\gamma$ is the representation $\Gamma$ on ${\mathcal H}$, defined by $(\Gamma x)(\mu)=\langle x, \gamma(\overline{\mu})\rangle,\ x\in{\mathcal H},\ \mu\in \Omega^{*}$.
Thus, $\Gamma$ is unitary and $ran\Gamma={\mathcal H}$.
Define function $K$ as $K(\mu, \lambda)=\Gamma(\gamma(\overline{\lambda}))(\mu)=\langle\gamma(\overline{\lambda}), \gamma(\overline{\mu})\rangle,\ \mu, \lambda\in\Omega^{*}$.
Then $$\Gamma T\Gamma^{*}K(\cdot, \lambda)=\Gamma T\Gamma^{*}\Gamma(\gamma(\overline{\lambda}))=\Gamma T(\gamma(\overline{\lambda}))=\Gamma\overline{\lambda}\gamma(\overline{\lambda})=\overline{\lambda}\Gamma\gamma(\overline{\lambda})=\overline{\lambda}K(\cdot, \lambda),\ \lambda\in \Omega^{*},$$
$K(\cdot, \lambda)$ is an eigenvector for $\Gamma T\Gamma^{*}$ with eigenvalue $\overline{\lambda}$.
Note that $\mathcal{H}=\overline{span}\{K(\cdot, \lambda),\ \lambda\in \Omega^{*}\}$, so $\Gamma T\Gamma^{*}$ is the adjoint of the multiplication operator
on a holomorphic function space ${\mathcal H}$ possessing a reproducing kernel $K$.
Without loss of generality, $T$ can be view as $M^{*}_{z}$, that is, $T=(M^{*}_{z}, {\mathcal H}, K)$.
More generally, for $T\in\mathcal{B}_{n}(\Omega)(n>1)$, similar results are proved,
$T$ can be realized as the adjoint $M^{*}_{z}$, where multiplication operator $M_{z}$
on a reproducing kernel Hilbert space of holomorphic $\mathbb{C}^{n}-$valued functions.

Suppose that two multiplication operators acting on two different $\mathbb{C}^{n}-$valued Hilbert spaces $\mathcal{H}_{1}, \mathcal{H}_{2}$ with reproducing kernels $K_{1}, K_{2}$, and $X: \mathcal{H}_{1}\rightarrow\mathcal{H}_{2}$ is a bounded operator intertwining these two operators.
We know that $\mathcal{H}_{i}=\overline{span}\{K_{i}(\cdot, \lambda)\xi, \lambda\in\Omega^{*}, \xi\in \mathbb{C}^{n}\}, i=1, 2$,
and $XK_{1}(\cdot, \lambda)\xi=K_{2}(\cdot, \lambda)\Psi(\lambda)\xi (\xi\in\mathbb{C}^{n}$)
for some $\mathcal{M}_{n}(\mathbb{C})-$valued function $\Psi$. Since $K_{i}(\cdot, \cdot), i=1, 2$ are anti-holomorphic in the second variable, it follows that $\Psi$ is anti-holomorphic.
In fact, $X^{*}$ is the multiplication operator $M_{\overline{\Psi}^{T}}$.
Furthermore, if two operators are unitarily equivalent, then there is a unitary operator $U$ intertwining them. Hence, $K_{1}(\cdot, \lambda)\xi=U^{*}K_{2}(\cdot, \lambda)\Phi^{*}(\lambda)\xi$, $U^{*}$ is multiplication operator $M_{\Phi}$,
where holomorphic function $\Phi$ satisfies for every $\mu$, $\Phi(\mu)$ is invertible.
It follows that $$K_{1}(\mu, \lambda)=\Phi(\mu)K_{2}(\mu, \lambda)\Phi^{*}(\lambda).$$
Conversely, if the equality above holds for two reproducing kernels $K_{1}, K_{2}$, then the corresponding operators are unitarily equivalent.
That is to say, the above equality is a sufficient and necessary condition for the unitary equivalence of the two operators
(See more details in \cite{CS2}).

\section{Unitary equivalence of operators acting on reproducing kernel Hilbert spaces}

In this section, by utilizing the reproducing kernels, we give another form of unitary equivalence of operators acting on reproducing kernel Hilbert spaces. By using this new form, we can also describe the structure of the intertwining matrix between two kernels related to the operators which are both reducible. Before introducing this theorem,  we need to describe the commutants of operators. 
Let $$\mathcal{A}'(T):=\{X|XT=TX\},\ \mathcal{A}'(T, T^{*}):=\{X|XT=TX,\ XT^{*}=T^{*}X\}.$$
$T\in{\mathcal L}({\mathcal H})$ is called irreducible, if the commutant of $T$ doesn't have any nontrivial self-adjoint idempotent operator,
equivalently, there is no non-trivial reducing subspace of $T$ (cf \cite{G5}, \cite{Hal3}). 
The irreducibility is a unitary invariant. In what follows, $T\in(IR)$ means that $T$ is an irreducible operator.
\begin{prop}\label{2111.151}
Let $T\in \mathcal{B}_{n}(\Omega)$. If $T=T_{1}\oplus T_{2} \oplus\cdots \oplus T_{k}$, then $k\leq n$ and $T_i$ is Cowen-Douglas operator for all $1\leq i\leq k$.
\label{pro1}
\end{prop}
\begin{proof}
If $T=T_{1}\oplus T_{2} \oplus\cdots \oplus T_{k}\in {\mathcal L}({\mathcal H})$, then $T_{i}\in{\mathcal L}({\mathcal H}_i)$ for non-zero Hilbert spaces ${\mathcal H_{i}}$, $1\leq i\leq k$.
Here $T_{i}$, $1\leq i\leq k$ may be reducible or irreducible.

According to Definition \ref{CD def}, we have ran$(T-\lambda)=\mathcal{H}$ for every $\lambda$ in $\Omega$. So
for every $y=(y_{1}, y_{2}, \cdots, y_{k})^{T}\in {\mathcal H}$, there exists $x=(x_{1}, x_{2}, \cdots, x_{k})^{T}\in {\mathcal H}$,
such that $(T-\lambda)x=y$, where $x_{i}, y_{i}\in {\mathcal H}_{i}$. By the arbitrariness of each component in $y$, we obtain that
for every $y_{i}\in {\mathcal H}_{i}$, there exist $x_{i}\in {\mathcal H}_{i}$ such that $(T_{i}-\lambda)x_{i}=y_{i},\ 1\leq i\leq k$.
Then ran$(T_{i}-\lambda)=\mathcal{H}_{i}$ for every $\lambda$ in $\Omega$, $1\leq i\leq k$.

Let $(z_{1}, z_{2}, \cdots, z_{k})^{T}\in\ker(T-\lambda)$. We have $z_{i}\in\ker(T_{i}-\lambda),1\leq i\leq k$. Conversely, if $z_{i}\in\ker(T_{i}-\lambda),1\leq i\leq k$, then
$(z_{1}, z_{2}, \cdots, z_{k})^{T}\in\ker(T-\lambda)$.
It follows that
\begin{equation}
\ker(T-\lambda)=\ker(T_{1}-\lambda)\oplus \ker(T_{2}-\lambda)\oplus \cdots \oplus \ker(T_{k}-\lambda).
\label{(1.2)}
\end{equation}
By Definition \ref{CD def} again, we have $\bigvee_{\lambda{\in}{\Omega}} \ker(T-\lambda)=\mathcal H$. Combining with equation (\ref{(1.2)}), we have
$$\begin{array}{lll}
{\mathcal H_{1}}\bigoplus {\mathcal H_{2}} \bigoplus\cdots \bigoplus {\mathcal H_{k}}
&=&\bigvee \limits_{\lambda{\in}{\Omega}} \ker(T-\lambda)\\
&=&\bigvee \limits_{\lambda{\in}{\Omega}}\big( \ker(T_{1}-\lambda)\bigoplus \ker(T_{2}-\lambda)\bigoplus \cdots \bigoplus \ker(T_{k}-\lambda)\big)\\
&=&\bigvee \limits_{\lambda{\in}{\Omega}}\ker(T_{1}-\lambda)\bigoplus \bigvee \limits_{\lambda{\in}{\Omega}}\ker(T_{2}-\lambda)\bigoplus\cdots\bigoplus \bigvee \limits_{\lambda{\in}{\Omega}}\ker(T_{k}-\lambda).
\end{array}$$
Therefore, $\bigvee \limits_{\lambda{\in}{\Omega}}\ker(T_{i}-\lambda)={\mathcal H_{i}},\ 1\leq i\leq k.$

Using formula (\ref{(1.2)}) again, we have dim\ $\ker(T-\lambda)=\sum_{i=1}^{k}$dim\;$\ker(T_{i}-\lambda)=n$ and dim\;$\ker(T_{i}-\lambda)$ is finite, $\lambda\in\Omega$. Thus, $T_{i}-\lambda$ is Fredholm operator and the index $\mbox{ind}(T_i-\lambda)=\mbox{dim}\,\ker (T_i-\lambda)\leq n$, since ran$(T_{i}-\lambda)=\mathcal{H}_{i}$ for $\lambda\in\Omega$ and $1\leq i\leq k$.
Setting $n_{i}:=\mbox{ind}(T_{i}-\lambda)$ for some $\lambda\in\Omega$, $1\leq i\leq k$. The continuity of index implies that for all $\mu$ in the same connected component of $\Omega$ containing $\lambda$, $n_{i}=\mbox{ind}(T_{i}-\lambda)$, $1\leq i\leq k$.
Claim $n_i>0, 1\leq i\leq k$. Otherwise, if there exist $i_0,1\leq i_0\leq k$ such that $n_{i_0}=0$, then $\mathcal{H}_{i_0}=\{0\}$. This contradicts to $\mathcal{H}_{i_0}$ being non-zero.
Thus we infer that $\Omega\subset\sigma(T_{i})$ and $T_{i}\in \mathcal{B}_{n_{i}}(\Omega)$, $1\leq i\leq k$.
\end{proof}

\begin{defn}\label{defe}
Let $T\in \mathcal{B}_{n}(\Omega)$, $T=(M^{*}_{z}, \mathcal{H}, K)$,
where $\mathcal{H}$ be the analytic function space with reproducing kernel $K$.
And $T$ has the following decomposition
$$T=\bigoplus\limits_{i=1}^{t}T_{i}^{(m_{i})},\ where \,\,T_{i}\in \mathcal{B}_{n_{i}}(\Omega),\ \sum\limits_{i=1}^{t}m_{i}n_{i}=n.$$
This decomposition is said to be (IR) decomposition, if $T_{i}\in (IR)$ and $T_{i}\nsim_{u} T_{j}(i\neq j)$.
\end{defn}
Therefore, from Proposition \ref{2111.151} it follows that there is a finite (IR) decomposition for each Cowen-Douglas operator. 
The following two lemmas characterize $\mathcal{A}'(T, T^{*})$ of operator $T$.
\begin{lem}\cite{FJW6}\label{Lyin}
If $T$ is irreducible on ${\mathcal H}$ and there is $X\in{\mathcal L}({\mathcal H})$ such that $X\in\mathcal{A}'(T, T^{*})$,
then $X$ is a scalar multiple of identity.
\end{lem}
Based on the lemma above, we have the following lemma, which describes $\mathcal{A}'(T, T^{*})$ when $T$ is reducible.
\begin{lem}\label{Lyin2}
If $T=\bigoplus_{i=1}^{t}T_{i}^{(m_{i})},\,T_{i}\in {\mathcal L}({\mathcal H}_{i})$ satisfies Definition \ref{defe},
and there exists $X\in{\mathcal L}({\mathcal H})$ such that $X\in\mathcal{A}'(T, T^{*})$, then
$X=diag[X_{1}, X_{2}, \cdots, X_{t}]$, where $X_{p}=(\!(c^{p}_{i,j}I_{\mathcal H_{p}})\!)_{i,j=1}^{m_{p}}$, $c^{p}_{i,j}$ are scalars, $1\leq p\leq t$.
In particular, $\mathcal{A}'(T, T^{*})\cong \sum\limits_{i=1}^t\bigoplus \mathcal{M}_{m_i}(\mathbb{C})$.
\end{lem}
\begin{proof}
For convenience, we write $T$ as the following form,
$T=\bigoplus\limits_{i=1}^{t}T_{i}^{(m_{i})}=\bigoplus\limits_{j=1}^{m}S_{j},\ \sum\limits_{i=1}^{t}m_{i}=m,$
where
$S_{1}=\cdots=S_{m_{1}}=T_{1}, S_{m_{1}+1}=\cdots=S_{m_{1}+m_{2}}=T_{2}, \cdots, S_{m_{1}+\cdots+m_{t-1}+1}=\cdots=S_{m}=T_{t}.$
Setting $X=(\!(X_{i,j})\!)_{i,j=1}^{m}$. By $X\in\mathcal{A}'(T, T^{*})$, we know that the size of $X_{i,j}$ is determined by $S_{i}$ and $S_{j}$, and
\begin{equation}\label{4.175}
X_{i,j}S_{j}=S_{i}X_{i,j},\  X_{i,j}S^{*}_{j}=S^{*}_{i}X_{i,j}.
\end{equation}
Next, combining these two formulas, we have
$$X^{*}_{i,j}X_{i,j}S_{j}=X^{*}_{i,j}S_{i}X_{i,j}=S_{j}X^{*}_{i,j}X_{i,j},\  X^{*}_{i,j}X_{i,j}S^{*}_{j}=S^{*}_{j}X^{*}_{i,j}X_{i,j};$$and
$$X_{i,j}X^{*}_{i,j}S^{*}_{i}=X_{i,j}S^{*}_{j}X^{*}_{i,j}=S^{*}_{i}X_{i,j}X^{*}_{i,j},\  X_{i,j}X^{*}_{i,j}S_{i}=S_{i}X_{i,j}X^{*}_{i,j}.$$
Since $S_{i}, S_{j}\in(IR)$, by Lemma \ref{Lyin}, there exist scalars $c_{i,j},d_{i,j}$ such that
\begin{equation}
X^{*}_{i,j}X_{i,j}=c_{i,j}I,\ X_{i,j}X^{*}_{i,j}=d_{i,j}I,\ 1\leq i,j\leq m.
\label{(h2)}
\end{equation}
Note that $X^{*}_{i,j}X_{i,j},X_{i,j}X^{*}_{i,j}$ are positive, we know that $c_{i,j},d_{i,j}\geq0$ for all $1\leq i,j\leq m$. From equation (\ref{(h2)}), we notice that
$X_{i,j}X^{*}_{i,j}X_{i,j}=d_{i,j}X_{i,j}=c_{i,j}X_{i,j},$
then $c_{i,j}=d_{i,j}$ for all $1\leq i, j\leq m$.

We assert that $c_{i,j}=0$, when $S_{i}\neq S_{j}$. By Definition \ref{defe},
we know that $S_{i}$ and $S_{j}$ are either equal or not unitarily equivalent. That is, $S_{i}\neq S_{j}$ means $S_{i}\nsim_{u} S_{j}$.
Suppose $c_{i,j}\neq0$, then $X_{i,j}X^{*}_{i,j}$, $X^{*}_{i,j}X_{i,j}$ are invertible.
Then equation (\ref{(h2)}) is equivalent to
$\frac{X_{i,j}}{\sqrt{c_{i,j}}}\big(\frac{X_{i,j}}{\sqrt{c_{i,j}}}\big)^{*}=\big(\frac{X_{i,j}}{\sqrt{c_{i,j}}}\big)^{*}\frac{X_{i,j}}{\sqrt{c_{i,j}}}=I.$
It follows that $\frac{X_{i,j}}{\sqrt{c_{i,j}}}$ is a unitary operator. By (\ref{4.175}), we obtain $S_{i}\sim_{u}S_{j}$.
It is a contradiction. Then the assertion is true.
In this case, it follows that $X^{*}_{i,j}X_{i,j}=0$.
By the polar decomposition of $X_{i,j}$, we have $X_{i,j}=V(X^{*}_{i,j}X_{i,j})^{\frac{1}{2}}$ for some partial isometry $V$.
Thus, we have $X_{i,j}=0$. 

In the case of $S_{i}= S_{j}$. By Lemma \ref{Lyin}, there exist $b_{i,j}$ such that $X_{i,j}=b_{i,j} I$ and $b_{i,j}\overline{b_{i,j}}=c_{i,j}$. Now, we know $X$ is quasi-diagonal.
To ensure uniformity of notation, we write $X$ in the form:
$X=diag[X_{1}, X_{2}, \cdots, X_{t}]$, where $X_{p}=(\!(c^{p}_{i,j}I_{\mathcal H_{p}})\!)_{i,j=1}^{m_{p}}$, $p=1, 2, \cdots, t$.
Clearly, if $X$ satisfies the above conditions, then $X\in\mathcal{A}'(T, T^{*})$.
Thus, $\mathcal{A}'(T, T^{*})\cong \sum\limits_{i=1}^t\bigoplus \mathcal{M}_{m_i}(\mathbb{C})$.
\end{proof}

A non negative definite kernel $K$ is said to be normalized at $\lambda_{0}\in\Lambda$, if
there exists a neighborhood $\Lambda_{0}$ of $\lambda_{0}$ in $\Lambda$ such that $K(\mu, \lambda_{0})=1$ for all $\mu\in \Lambda_{0}$ \cite{CS2}.
Inspired by Theorem \ref{L1} due to R. E. Curto and N. Salinas, we have the following theorem.
\begin{thm}\label{KK1}
Let $T,S\in \mathcal{B}_{n}(\Omega)$, and $T=(M^{*}_{z}, \mathcal{H}_{T}, K_{T}), S=(M^{*}_{z}, \mathcal{H}_{S}, K_{S})$,
where $\mathcal{H}_{T}, \mathcal{H}_{S}$ are the analytic function spaces of reproducing kernels $K_{T},K_{S}$, respectively.
Then $T\sim_{u}S$ if and only if there is a holomorphic matrix-valued function $\Psi$ and $\Psi(\mu)$ is invertible for every $\mu\in\Lambda$, a open subset of $\Omega^{*}$, such that
$$\frac{\partial}{\partial\lambda}[K_{T}^{-1}(\lambda, \lambda)\Psi(\lambda)^{-1}K_{S}(\lambda, \lambda)]=0$$for all $\lambda\in\Lambda.$
\end{thm}
\begin{proof}
Necessity: If $T\sim_{u}S$, by Theorem \ref{L1}, there is a holomorphic and invertible matrix-valued function $\Psi$ such that
$K_{S}(\mu, \lambda)=\Psi(\mu)K_{T}(\mu, \lambda)\Psi^{*}(\lambda)$, equivalently, we have
$$\Psi^{*}(\lambda)=K_{T}^{-1}(\mu, \lambda)\Psi(\mu)^{-1}K_{S}(\mu, \lambda)$$ for all $\mu, \lambda\in\Lambda$, a open subset of $\Omega^{*}$.
Since $\Psi^{*}(\lambda)$ is anti-holomorphic, we obtain
$$\frac{\partial}{\partial\lambda}[K_{T}^{-1}(\lambda, \lambda)\Psi(\lambda)^{-1}K_{S}(\lambda, \lambda)]=0,\ \lambda\in\Lambda.$$

Sufficiency: If $\frac{\partial}{\partial\lambda}[K_{T}^{-1}(\lambda, \lambda)\Psi(\lambda)^{-1}K_{S}(\lambda, \lambda)]=0,\ \lambda\in\Lambda$, then we know that it is anti-holomorphic for the variable $\lambda$.
So we can find a holomorphic matrix-valued invertible function $\Phi$ such that
$$\Phi^{*}(\lambda)=K_{T}^{-1}(\lambda, \lambda)\Psi(\lambda)^{-1}K_{S}(\lambda, \lambda),\ \lambda\in\Lambda,$$ that is,
$K_{S}(\lambda, \lambda)=\Psi(\lambda)K_{T}(\lambda, \lambda)\Phi^{*}(\lambda),\ \lambda\in\Lambda$.
By Proposition 1 in \cite{EM}, we have
\begin{equation}\label{4.111}
K_{S}(\mu, \lambda)=\Psi(\mu)K_{T}(\mu, \lambda)\Phi^{*}(\lambda),\ \mu, \lambda\in\Lambda.
\end{equation}

{\bf Claim:} If there are $\mathcal{M}_{n}(\mathbb{C})-$ valued holomorphic and invertible functions $\Phi, \Psi$
such that equation (\ref{4.111}) is valid, then $T\sim_{u}S$.

Conjugating equation (\ref{4.111}) and then interchanging the variables $\mu, \lambda$, one arrives at the
equation $K_{S}(\mu, \lambda)=\Phi(\mu)K_{T}(\mu, \lambda)\Psi^{*}(\lambda),\ \mu, \lambda\in\Lambda$.
It follows that
$\Psi(\mu)K_{T}(\mu, \lambda)\Phi^{*}(\lambda)=\Phi(\mu)K_{T}(\mu, \lambda)\Psi^{*}(\lambda)$, $\mu, \lambda\in\Lambda.$ That means
$\Phi^{-1}(\mu)\Psi(\mu)K_{T}(\mu, \lambda)=K_{T}(\mu, \lambda)(\Phi^{-1}(\lambda)\Psi(\lambda))^{*},\ \mu, \lambda\in\Lambda.$
Setting $\mathcal{X}=\Phi^{-1}\Psi$. We have
\begin{equation}\label{4.121}
\mathcal{X}(\mu)K_{T}(\mu, \lambda) =K_{T}(\mu, \lambda)\mathcal{X}^{*}(\lambda),\ \mu, \lambda\in\Lambda.
\end{equation}
\textbf{Case 1:} Suppose $T$ is irreducible. Let
\begin{equation}\label{215.208}
K_{0}(\mu, \lambda):=K_{T}^{\frac{1}{2}}(\lambda_{0}, \lambda_{0})K_{T}^{-1}(\mu, \lambda_{0})K_{T}(\mu, \lambda)(K_{T}^{-1}(\lambda, \lambda_{0}))^{*}K_{T}^{\frac{1}{2}}(\lambda_{0}, \lambda_{0}),\ \mu, \lambda, \lambda_{0}\in\Lambda,
\end{equation}
It's easy to see that $K_{0}(\mu, \lambda)$ is normalized at $\lambda_{0}$.
Suppose that $\widetilde{M}_{z}$ is a multiplication operator on the Hilbert space $\mathcal{H}_{K_{0}}$ with a reproducing kernel $K_{0}$.
Then
\begin{equation}\label{215.203}
\mathcal{H}_{K_{0}}=\bigvee _{\lambda{\in}{\Lambda}^{*}} \ker(\widetilde{M}_{z}^{*}-\lambda)=\bigvee _{\lambda{\in}{\Lambda}^{*}} \{K_{0}(\cdot,\lambda)\xi,\ \xi\in \mathbb{C}^{n}\}.
\end{equation}
From equation (\ref{4.121}), we can imply
\begin{equation}\label{215.201}
\mathcal{Y}(\mu)K_{0}(\mu, \lambda)=K_{0}(\mu, \lambda)\mathcal{Y}^{*}(\lambda),\ \mu, \lambda\in\Lambda,
\end{equation}
where
\begin{equation}\label{215.207}
\mathcal{Y}(\mu)=K_{T}^{\frac{1}{2}}(\lambda_{0}, \lambda_{0})K_{T}^{-1}(\mu, \lambda_{0})\mathcal{X}(\mu)K_{T}(\mu, \lambda_{0})K_{T}^{-\frac{1}{2}}(\lambda_{0}, \lambda_{0}).
\end{equation}
We can see that $\mathcal{Y}$ is a holomorphic matrix.
In equation (\ref{215.201}), setting $\mu=\lambda=\lambda_{0}$, we come to $\mathcal{Y}(\lambda_{0})=\mathcal{Y}^{*}(\lambda_{0})$. This implies that $\mathcal{Y}(\lambda_{0})$ is a constant matrix. For any $\mu\in\Lambda$, setting $\lambda=\lambda_{0}$, we also come to $\mathcal{Y}(\mu)=\mathcal{Y}^{*}(\lambda_{0})$.
Thus, $\mathcal{Y}(\cdot)$ is a constant matrix. We define an operator $Y: \mathcal{H}_{K_{0}}\rightarrow\mathcal{H}_{K_{0}}$ as
\begin{equation}\label{215.202}
(Yf)(\cdot)=\mathcal{Y}f(\cdot)\ \ \mbox{for all}\ \ f\in \mathcal{H}_{K_{0}}.
\end{equation}
We can see that
$\widetilde{M}_{z}Yf(z)
=\widetilde{M}_{z}\mathcal{Y}f(z)=z\mathcal{Y}f(z)
=\mathcal{Y}(zf(z))=\mathcal{Y}\widetilde{M}_{z}(f(z))=Y\widetilde{M}_{z}f(z),$
that is,
\begin{equation}\label{215.204}
\widetilde{M}_{z}Y=Y\widetilde{M}_{z}.
\end{equation}
For equation (\ref{215.202}), in particular, we have
\begin{equation}\label{215.205}
YK_{0}(\cdot, \lambda)\xi=\mathcal{Y}K_{0}(\cdot, \lambda)\xi
\end{equation}
for all $\lambda\in\Lambda$ and $\xi\in \mathbb{C}^{n}$.
We will calculate the adjoint of the operator $Y$. We have
$$\begin{array}{lll}
\langle YK_{0}(\cdot, \mu)\eta, K_{0}(\cdot, \lambda)\xi\rangle
&=&\langle\mathcal{Y}K_{0}(\cdot, \mu)\eta, K_{0}(\cdot, \lambda)\xi\rangle=\langle\mathcal{Y}K_{0}(\lambda, \mu)\eta, \xi\rangle\\
&=&\langle K_{0}(\lambda, \mu)\eta, \mathcal{Y}\xi\rangle=\langle K_{0}(\cdot, \mu)\eta, K_{0}(\cdot, \lambda)\mathcal{Y}\xi\rangle
\end{array}$$
for all $\lambda, \mu\in\Lambda$ and $\xi, \eta\in \mathbb{C}^{n}$. In other words,
\begin{equation}\label{215.206}
Y^{*}K_{0}(\cdot, \lambda)\xi=K_{0}(\cdot, \lambda)\mathcal{Y}\xi
\end{equation}
for all $\lambda\in\Lambda$ and $\xi\in \mathbb{C}^{n}$ (see \cite{KG} for more details).

Then, by equations (\ref{215.203}),(\ref{215.201}),(\ref{215.205}) and (\ref{215.206}), we obtain that $Y$ is a self-adjoint operator.
Combining with (\ref{215.204}), we can imply $Y\in \mathcal{A}'(\widetilde{M}_{z}, \widetilde{M}_{z}^{*})$.
In the special case of Theorem 8 in \cite{C12}, we can see $Y$ is bounded.
Note that $K_{T}^{\frac{1}{2}}(\lambda_{0}, \lambda_{0})K_{T}^{-1}(\cdot, \lambda_{0})$ is holomorphic and non-zero on $\Lambda$ in equation (\ref{215.208}),
it follows that $M_z^*$ defined on $(\mathcal{H}_T,K_T)$ is unitarily equivalent to $\widetilde{M}_{z}$ on $(\mathcal{H}_0,K_0)$ by Theorem \ref{L1}. Then we have $\widetilde{M}_{z}\in(IR)$, since unitary transformation doesn't change the irreducibility of operators.
According to Lemma \ref{Lyin} and $Y$ is self-adjoint, we know that $Y=cI, c\in \mathbb{R}$,
$I$ is the identity of $\mathcal{H}_{0}$. Thus, from equations (\ref{215.207}) and (\ref{215.202}), we have $\mathcal{X}=cI_{n\times n}, c\in \mathbb{R}$. Since $\mathcal{X}=\Phi^{-1}\Psi$, we obtain $K_{S}(\mu, \lambda)=c\Phi(\mu)K_{T}(\mu, \lambda)\Phi(\lambda)$.
From the positive definiteness of $K_{S}(\mu, \lambda)$ and $\Phi(\mu)K_{T}(\mu, \lambda)\Phi(\lambda)$, we can get $c>0$ and
$K_{S}(\mu, \lambda)=(\sqrt{c}\Phi(\mu))K_{T}(\mu, \lambda)(\sqrt{c}\Phi(\lambda)).$ Hence, $T\sim_{u}S$.

\textbf{Case 2:} Suppose $T$ is reducible. Without losing generality, we suppose that
$$T=\bigoplus\limits_{i=1}^{t}T_{i}^{(m_{i})}\ and \,\,T_{i}\nsim_{u} T_{j}(i\neq j),$$
where $\mathcal{H}_{i}$ are the analytic function spaces of reproducing kernels $K_{i}$, respectively.
We write $T$ as $T=\bigoplus_{i=1}^{m}S_{i}$ and $S_{i}=(M^{*}_{z}, \mathcal{H}_{i}, K_{i})$, where
$$S_{1}=\cdots=S_{m_{1}}=T_{1}, S_{m_{1}+1}=\cdots=S_{m_{1}+m_{2}}=T_{2}, \cdots, S_{m_{1}+\cdots+m_{t-1}+1}=\cdots=S_{m}=T_{t}.$$
Setting $\mathcal{X}=(\!(\mathcal{X}_{i,j})\!)_{i,j=1}^{m}$. Then $\mathcal{X}(\mu)K_{T}(\mu, \lambda)=K_{T}(\mu, \lambda)\mathcal{X}^{*}(\lambda),\ \mu, \lambda\in\Lambda$ is equivalent to
\begin{equation}\label{215.211}
\mathcal{X}_{i,j}(\mu)K_{j}(\mu, \lambda)=K_{i}(\mu, \lambda)\mathcal{X}_{j,i}^{*}(\lambda),\ \mathcal{X}_{j,i}(\mu)K_{i}(\mu, \lambda)=K_{j}(\mu, \lambda)\mathcal{X}_{i,j}^{*}(\lambda).
\end{equation}
It follows that
\begin{equation}\label{215.212}
\begin{array}{lll}

& &\mathcal{X}_{j,i}(\mu)\mathcal{X}_{i,j}(\mu)K_{j}(\mu, \lambda)=K_{j}(\mu, \lambda)(\mathcal{X}_{j,i}(\lambda)\mathcal{X}_{i,j}(\lambda))^{*},\\
& &\mathcal{X}_{i,j}(\mu)\mathcal{X}_{j,i}(\mu)K_{i}(\mu, \lambda)=K_{i}(\mu, \lambda)(\mathcal{X}_{i,j}(\lambda)\mathcal{X}_{j,i}(\lambda))^{*}
\end{array}
\end{equation}
for all $\mu, \lambda\in\Lambda$.
 By Case 1, we have
$\mathcal{X}_{j,i}\mathcal{X}_{i,j}=c_{j,i}I,\ \mathcal{X}_{i,j}\mathcal{X}_{j,i}=c_{i,j}I$ and $c_{j,i}=c_{i,j}\in \mathbb{R}$.

If there exist $i,j$, such that $c_{i,j}\neq0$, then $\mathcal{X}_{i,j},\mathcal{X}_{j,i}$ are invertible. By equation (\ref{215.211}), we have
$K_{j}(\mu, \lambda)=\frac{1}{c_{i,j}}\mathcal{X}_{j,i}(\mu)K_{i}(\mu, \lambda)\mathcal{X}_{j,i}^{*}(\lambda)$ and $c_{i,j}>0$. Therefore,
$K_{j}(\mu, \lambda)=\frac{\mathcal{X}_{j,i}(\mu)}{\sqrt{c_{i,j}}}K_{i}(\mu, \lambda)(\frac{\mathcal{X}_{j,i}(\lambda)}{\sqrt{c_{i,j}}})^{*},\ \mu, \lambda\in\Lambda.$
From Theorem \ref{L1}, we obtain $S_{i}\sim_{u}S_{j}$, that is, $S_{i}=S_{j}$, since $S_{i}$ and $S_{j}$ are either equal or not unitarily equivalent.
In the following, we define an operator $U_{ji}: \mathcal{H}_{i}\rightarrow\mathcal{H}_{j}$, $U_{ji}f(\cdot)=\frac{\mathcal{X}_{j,i}(\cdot)}{\sqrt{c_{i,j}}}f(\cdot)$ for all $f\in \mathcal{H}_{i}$.
By Lemma 3.9 of \cite{CS2}, we know $U_{ji}$ is a unitary transformation. From the definition of $U_{ji}$, we can also see $U_{ji}M_{z}=M_{z}U_{ji}$.
It follows that $U_{ji}M_{z}^{*}=M_{z}^{*}U_{ji}$, where $M_{z}$ is the multiplication operator on Hilbert space $\mathcal{H}_{i}$ with reproducing kernels $K_{i}$.
According to Lemma \ref{Lyin}, we have that $U_{ji}$ is $e^{i\theta_{ji}}$ multiple of identity, $\theta_{ji}\in\mathbb{R}$. Then $\mathcal{X}_{j,i}=\sqrt{c_{i,j}}e^{i\theta_{ji}}I$.
Similarly, we can find unitary $U_{ij}=e^{i\theta_{ij}}I$ such that $\mathcal{X}_{i,j}=\sqrt{c_{i,j}}e^{i\theta_{ij}}I$.
By equation (\ref{215.211}), we have $\mathcal{X}_{j,i}=\mathcal{X}_{i,j}^{*}$.

Based on the above analysis, we know that if $S_{i}\nsim_{u}S_{j}$, then $c_{i,j}=c_{j,i}=0$. Notice that $K_{i}(\lambda,\lambda)$ is positive definite and invertible, therefore, there exists
$H_{i}$ such that $K_{i}(\lambda,\lambda)=H_{i}(\lambda)H_{i}^{*}(\lambda),\ \lambda\in\Lambda$.
For equations (\ref{215.211}) and (\ref{215.212}), setting $\mu:=\lambda$, we have
$$0=\mathcal{X}_{j,i}(\lambda)\mathcal{X}_{i,j}(\lambda)K_{j}(\lambda, \lambda)=\mathcal{X}_{j,i}(\lambda)K_{i}(\lambda, \lambda)\mathcal{X}^{*}_{j,i}(\lambda)
=(\mathcal{X}_{j,i}(\lambda)H_{i}(\lambda))(\mathcal{X}_{j,i}(\lambda)H_{i}(\lambda))^{*},\ \lambda\in\Lambda.$$
That is, $\mathcal{X}_{j,i}(\lambda)=0$ for all $\lambda\in\Lambda$.
Hence, $\mathcal{X}$ is symmetric, quasi-diagonal and its elements are scalars.

By the characteristics of $\mathcal{X}$ discussed above and $\Psi=\Phi\mathcal{X}$, we have
$$K_{T}(\mu, \lambda)=\Psi^{-1}(\mu)K_{S}(\mu, \lambda)(\Phi^{-1}(\lambda))^{*}=\mathcal{X}^{-1}\Phi^{-1}(\mu)K_{S}(\mu, \lambda)(\Phi^{-1}(\lambda))^{*}.$$
Conjugating the last equation and then interchanging the variables $\mu, \lambda$, one arrives at the
equation $K_{T}(\mu, \lambda)=\Phi^{-1}(\mu)K_{S}(\mu, \lambda)(\Phi^{-1}(\lambda))^{*}\mathcal{X}^{-1}$.
It follows that $\mathcal{X}^{-1}$ is positive definite. By functional calculus that (for any positive matrix $A$, $A^{\frac{1}{2}}$ can be approximated in norm by polynomials about $A$)
$$\mathcal{X}^{-\frac{1}{2}}\Phi^{-1}(\mu)K_{S}(\mu, \lambda)(\Phi^{-1}(\lambda))^{*}=\Phi^{-1}(\mu)K_{S}(\mu, \lambda)(\Phi^{-1}(\lambda))^{*}\mathcal{X}^{-\frac{1}{2}}.$$
Hence, we can imply
$K_{T}(\mu, \lambda)=(\mathcal{X}^{-\frac{1}{2}}\Phi^{-1}(\mu))K_{S}(\mu, \lambda)(\mathcal{X}^{-\frac{1}{2}}\Phi^{-1}(\lambda))^{*}$.
By Theorem \ref{L1} again, we have $T\sim_{u}S$.
\end{proof}
The following corollary is a direct consequence of Theorem \ref{KK1}.
\begin{cor}
Let $T,S\in \mathcal{B}_{n}(\Omega)$, and $T=(M^{*}_{z}, \mathcal{H}_{T}, K_{T}), S=(M^{*}_{z}, \mathcal{H}_{S}, K_{S})$,
where $\mathcal{H}_{T}, \mathcal{H}_{S}$ are the analytic function spaces of reproducing kernels $K_{T},K_{S}$, respectively.
Then $T\sim_{u}S$ if and only if there is a holomorphic and invertible matrix-valued function $\Phi$ on $\Lambda$, a open subset of $\Omega^{*}$, such that connection matrices $\Theta_{T}, \Theta_{S}$ of $E_{T}, E_{S}$ satisfy
$\Theta_{T}=\Phi^{-1}\Theta_{S}\Phi+\Phi^{-1}\Phi'.$
\end{cor}
\begin{proof}
By Theorem \ref{KK1}, we know that $T\sim_{u}S$ if and only if
$\frac{\partial}{\partial\lambda}[K_{T}^{-1}(\lambda, \lambda)\Psi(\lambda)^{-1}K_{S}(\lambda, \lambda)]=0$ for some holomorphic matrix-valued function $\Psi$ and all $\lambda\in\Lambda$, a open subset of $\Omega^{*}$.
And, if we take the partial derivative of both sides of equation $$K_{T}(\lambda, \lambda)[K_{T}^{-1}(\lambda, \lambda)\Psi^{-1}(\lambda)K_{S}(\lambda, \lambda)]=\Psi^{-1}(\lambda)K_{S}(\lambda, \lambda),\ \lambda\in\Lambda,$$ then
$\frac{\partial}{\partial\lambda}K_{T}(\lambda, \lambda)[K_{T}^{-1}(\lambda, \lambda)\Psi^{-1}(\lambda)K_{S}(\lambda, \lambda)]=\frac{\partial}{\partial\lambda}(\Psi^{-1}(\lambda)K_{S}(\lambda, \lambda))$. It follows that
$$\begin{array}{lll}
\frac{\partial}{\partial\lambda}K_{T}(\lambda, \lambda)K_{T}^{-1}(\lambda, \lambda)
&=&\frac{\partial}{\partial\lambda}(\Psi^{-1}(\lambda)K_{S}(\lambda, \lambda))[\Psi^{-1}(\lambda)K_{S}(\lambda, \lambda)]^{-1}\\
&=&\Psi^{-1}(\lambda)\frac{\partial}{\partial\lambda}K_{S}(\lambda, \lambda)K_{S}^{-1}(\lambda, \lambda)\Psi(\lambda)+(\Psi^{-1}(\lambda))'\Psi(\lambda)
\end{array}$$
for all $\lambda\in\Lambda.$
Since $h_{T}(\lambda)=K_{T}(\overline{\lambda}, \overline{\lambda}),\ h_{S}(\lambda)=K_{S}(\overline{\lambda}, \overline{\lambda})$, we also have
$$\frac{\partial}{\partial\overline{\lambda}}h_{T}(\lambda)h_{T}^{-1}(\lambda)
=\overline{\Psi(\lambda)}^{-1}\frac{\partial}{\partial\overline{\lambda}}h_{S}(\lambda)h_{S}^{-1}(\lambda)\overline{\Psi(\lambda)}+(\overline{\Psi(\lambda)}^{-1})'\overline{\Psi(\lambda)},\lambda\in\Lambda^{*}\subset\Omega.$$
Conjugating the last equation, we obtain
$$h_{T}^{-1}(\lambda)\frac{\partial}{\partial\lambda}h_{T}(\lambda)=\Psi^{T}(\lambda)h_{S}^{-1}(\lambda)\frac{\partial}{\partial\lambda}h_{S}(\lambda)(\Psi^{T}(\lambda))^{-1}+\Psi^{T}(\lambda)((\Psi^{T}(\lambda))^{-1})',\ \lambda\in\Lambda^{*}.$$
Setting $\Phi:=(\Psi^{T})^{-1}$. We can imply that
$$h_{T}^{-1}(\lambda)\frac{\partial}{\partial\lambda}h_{T}(\lambda)=\Phi^{-1}(\lambda)h_{S}^{-1}(\lambda)\frac{\partial}{\partial\lambda}h_{S}(\lambda)\Phi(\lambda)+\Phi^{-1}(\lambda)\Phi'(\lambda),\ \lambda\in\Lambda^{*},$$
that is, $\Theta_{T}=\Phi^{-1}\Theta_{S}\Phi+\Phi^{-1}\Phi'.$
\end{proof}
From the proof of Theorem \ref{KK1}, when changing equation (\ref{(4.8990)}) into equation (\ref{4.111}),
we can also obtain $T$ is unitarily equivalent to $S$.
In the case of equation (\ref{4.111}), it can be seen from the following proposition that
the corresponding  curvatures and covariant partial derivatives of two bundles are also similar. Furthermore, the intertwining matrix is determined only by $\Phi$.
\begin{prop}\label{pro4.1}
Let $T,S\in \mathcal{B}_{n}(\Omega)$, and $T=(M^{*}_{z}, \mathcal{H}_{T}, K_{T}), S=(M^{*}_{z}, \mathcal{H}_{S}, K_{S})$,
where $\mathcal{H}_{T}, \mathcal{H}_{S}$ are the analytic function spaces of reproducing kernels $K_{T},K_{S}$, respectively.
If there exist $\mathcal{M}_{n}(\mathbb{C})-$ valued holomorphic functions $\Phi, \Psi$ and $\Phi(\mu), \Psi(\mu)$ are invertible for every $\mu\in \Omega^{*}$, such that
\begin{equation}\label{(3.1)}
K_{S}(\mu, \lambda)=\Psi(\mu)K_{T}(\mu, \lambda)\Phi^{*}(\lambda),\ \mu, \lambda\in \Omega^{*},
\end{equation}
then $\Phi^{T}\mathcal{K}_{S,\lambda^{i}\overline{\lambda}^{j}}=\mathcal{K}_{T,\lambda^{i}\overline{\lambda}^{j}}\Phi^{T},\ i,j\geq 1.$
\end{prop}
\begin{proof}
From formula (\ref{(3.1)}), we can imply that
$$\begin{array}{lll}
h_{S}(\lambda)
&=&K_{S}(\overline{\lambda}, \overline{\lambda})\\
&=&\Psi(\overline{\lambda})K_{T}(\overline{\lambda}, \overline{\lambda})\overline{\Phi^{T}(\overline{\lambda})}\\
&=&\Psi(\overline{\lambda})h_{T}(\lambda)\Phi^{T}(\lambda),
\end{array}$$
where $h_{T}, h_{S}$ are the metrics of $E_{T}, E_{S}$, respectively.
Since $\Psi$ and $\Phi$ are analytic invertible matrixes, we have the following equations
$$\frac{\partial}{\partial\overline{\lambda}}\Phi^{T}(\lambda)=\frac{\partial}{\partial\overline{\lambda}}(\Phi^{T})^{-1}(\lambda)=\frac{\partial}{\partial\overline{\lambda}}(\Phi^{T})'(\lambda)=0,
\ \ \frac{\partial}{\partial\lambda}\Psi(\overline{\lambda})=0.$$
The curvatures corresponding to $h_{T}$ and $h_{S}$ are $\mathcal{K}_{T}$ and $\mathcal{K}_{S}$, respectively. Thus, we have that
$$\begin{array}{lll}
\mathcal{K}_{S}
&=&-\frac{\partial}{\partial\overline{\lambda}}(h^{-1}_{S}\frac{\partial}{\partial\lambda} h_{S})\\
&=&-\frac{\partial}{\partial\overline{\lambda}}[(\Phi^{T})^{-1}h^{-1}_{T}\overline{\Psi}^{-1}\frac{\partial}{\partial\lambda} (\overline{\Psi}h_{T}\Phi^{T})]\\
&=&-\frac{\partial}{\partial\overline{\lambda}}[(\Phi^{T})^{-1}h^{-1}_{T}((\frac{\partial}{\partial\lambda} h_{T})\Phi^{T}+h_{T}(\Phi^{T})^{'})]\\
&=&-\frac{\partial}{\partial\overline{\lambda}}[(\Phi^{T})^{-1}h^{-1}_{T}(\frac{\partial}{\partial\lambda} h_{T})\Phi^{T}+(\Phi^{T})^{-1}(\Phi^{T})^{'}]\\
&=&-\frac{\partial}{\partial\overline{\lambda}}((\Phi^{T})^{-1}h^{-1}_{T}(\frac{\partial}{\partial\lambda} h_{T})\Phi^{T})\\
&=&-(\Phi^{T})^{-1}\frac{\partial}{\partial\overline{\lambda}}(h^{-1}_{T}\frac{\partial}{\partial\lambda} h_{T})\Phi^{T}\\
&=&(\Phi^{T})^{-1}\mathcal{K}_{T}\Phi^{T}.
\end{array}$$

Using the definition of the covariant partial derivatives of curvature, we obtain that
$$\mathcal{K}_{S,\overline{\lambda}}=\frac{\partial}{\partial\overline{\lambda}}\mathcal{K}_{S}=\frac{\partial}{\partial\overline{\lambda}}((\Phi^{T})^{-1}\mathcal{K}_{T}\Phi^{T})
=(\Phi^{T})^{-1}(\frac{\partial}{\partial\overline{\lambda}}\mathcal{K}_{T})\Phi^{T}=(\Phi^{T})^{-1}\mathcal{K}_{T,\overline{\lambda}}\Phi^{T}.$$
By repeating the above operation, we obtain
$$\mathcal{K}_{S,\overline{\lambda}^{2}}=(\Phi^{T})^{-1}\mathcal{K}_{T,\overline{\lambda}^{2}}\Phi^{T},\
\mathcal{K}_{S,\overline{\lambda}^{3}}=(\Phi^{T})^{-1}\mathcal{K}_{T,\overline{\lambda}^{3}}\Phi^{T}, \cdots$$
It follows that
\begin{equation}\label{(1)}
\mathcal{K}_{S,\overline{\lambda}^{j}}=(\Phi^{T})^{-1}\mathcal{K}_{T,\overline{\lambda}^{j}}\Phi^{T},\ j\geq1.
\end{equation}
Considering
$(\Phi^{T})^{-1}\Phi^{T}=I$, we have $((\Phi^{T})^{-1}\Phi^{T})'=0$, that is,
$\widetilde{\Phi}\Phi^{T}+(\Phi^{T})^{-1}(\Phi^{T})'=0$ and
\begin{equation}
\widetilde{\Phi}+(\Phi^{T})^{-1}(\Phi^{T})'(\Phi^{T})^{-1}=0,
\label{(3)}
\end{equation}
where $\widetilde{\Phi}=((\Phi^{T})^{-1})'$. Next, we will prove the following formula
\begin{equation}
\mathcal{K}_{S,\lambda^{i}}=(\Phi^{T})^{-1}\mathcal{K}_{T,\lambda^{i}}\Phi^{T},\ i\geq1.
\label{(2)}
\end{equation}
When $i=1$, from the equation (\ref{(3)}) and the definition of the covariant partial derivatives of curvature again, we have
$$\begin{array}{lll}
\mathcal{K}_{S,\lambda}
&=&\frac{\partial}{\partial\lambda}\mathcal{K}_{S}+h^{-1}_{S}(\frac{\partial}{\partial\lambda}h_{S})\mathcal{K}_{S}-\mathcal{K}_{S}h^{-1}_{S}\frac{\partial}{\partial\lambda}h_{S}\\
&=&\frac{\partial}{\partial\lambda}((\Phi^{T})^{-1}\mathcal{K}_{T}\Phi^{T})+(\overline{\Psi}h_{T}\Phi^{T})^{-1}\frac{\partial}{\partial\lambda}(\overline{\Psi}h_{T}\Phi^{T})((\Phi^{T})^{-1}\mathcal{K}_{T}\Phi^{T})\\
& &-((\Phi^{T})^{-1}\mathcal{K}_{T}\Phi^{T})(\overline{\Psi}h_{T}\Phi^{T})^{-1}\frac{\partial}{\partial\lambda}(\overline{\Psi}h_{T}\Phi^{T})\\
&=&\frac{\partial}{\partial\lambda}((\Phi^{T})^{-1}\mathcal{K}_{T}\Phi^{T})+(\Phi^{T})^{-1}h_{T}^{-1}\frac{\partial}{\partial\lambda}(h_{T}\Phi^{T})((\Phi^{T})^{-1}\mathcal{K}_{T}\Phi^{T})
-((\Phi^{T})^{-1}\mathcal{K}_{T})h_{T}^{-1}\frac{\partial}{\partial\lambda}(h_{T}\Phi^{T})\\
&=&[\widetilde{\Phi}+(\Phi^{T})^{-1}(\Phi^{T})'(\Phi^{T})^{-1}]\mathcal{K}_{T}\Phi^{T}+(\Phi^{T})^{-1}[\frac{\partial}{\partial\lambda}\mathcal{K}_{T}+h_{T}^{-1}(\frac{\partial}{\partial\lambda}h_{T})\mathcal{K}_{T}-\mathcal{K}_{T}h_{T}^{-1}\frac{\partial}{\partial\lambda}h_{T}]\Phi^{T}\\
&=&(\Phi^{T})^{-1}\mathcal{K}_{T,\lambda}\Phi^{T}.
\end{array}$$
We then complete the proof of the formula (\ref{(2)}) by induction. Suppose that it is valid
for any $i\leq k$. For $i=k+1$,
with the assumptions and the equation (\ref{(3)}), we have
$$\begin{array}{lll}
\mathcal{K}_{S,\lambda^{k+1}}
&=&\frac{\partial}{\partial\lambda}\mathcal{K}_{S,\lambda^{k}}+h^{-1}_{S}(\frac{\partial}{\partial\lambda}h_{S})\mathcal{K}_{S,\lambda^{k}}-\mathcal{K}_{S,\lambda^{k}}h^{-1}_{S}\frac{\partial}{\partial\lambda}h_{S}\\
&=&\frac{\partial}{\partial\lambda}((\Phi^{T})^{-1}\mathcal{K}_{T,\lambda^{k}}\Phi^{T})+(\overline{\Psi}h_{T}\Phi^{T})^{-1}\frac{\partial}{\partial\lambda}(\overline{\Psi}h_{T}\Phi^{T})((\Phi^{T})^{-1}\mathcal{K}_{T,\lambda^{k}}\Phi^{T})\\
& &-((\Phi^{T})^{-1}\mathcal{K}_{T,\lambda^{k}}\Phi^{T})(\overline{\Psi}h_{T}\Phi^{T})^{-1}\frac{\partial}{\partial\lambda}(\overline{\Psi}h_{T}\Phi^{T})\\
&=&\frac{\partial}{\partial\lambda}((\Phi^{T})^{-1}\mathcal{K}_{T,\lambda^{k}}\Phi^{T})+(\Phi^{T})^{-1}h_{T}^{-1}\frac{\partial}{\partial\lambda}(h_{T}\Phi^{T})((\Phi^{T})^{-1}\mathcal{K}_{T,\lambda^{k}}\Phi^{T})
-((\Phi^{T})^{-1}\mathcal{K}_{T,\lambda^{k}})h_{T}^{-1}\frac{\partial}{\partial\lambda}(h_{T}\Phi^{T})\\
&=&[\widetilde{\Phi}+(\Phi^{T})^{-1}(\Phi^{T})'(\Phi^{T})^{-1}]\mathcal{K}_{T,\lambda^{k}}\Phi^{T}+(\Phi^{T})^{-1}[\frac{\partial}{\partial\lambda}\mathcal{K}_{T,\lambda^{k}}+h_{T}^{-1}(\frac{\partial}{\partial\lambda}h_{T})\mathcal{K}_{T,\lambda^{k}}
-\mathcal{K}_{T,\lambda^{k}}h_{T}^{-1}\frac{\partial}{\partial\lambda}h_{T}]\Phi^{T}\\
&=&(\Phi^{T})^{-1}\mathcal{K}_{T,\lambda^{k+1}}\Phi^{T}.
\end{array}$$
Hence $i=k+1$ holds as well. That is, equation (\ref{(2)}) is true.

Lastly, by equations (\ref{(1)}) and (\ref{(2)}), we obtain
\begin{equation}\label{4.99}
\begin{array}{lll}
\mathcal{K}_{S,\lambda^{i}\overline{\lambda}^{j}}
&=&\frac{\partial^{j}}{\partial\overline{\lambda}^{j}}\mathcal{K}_{S,\lambda^{i}}\\
&=&\frac{\partial^{j}}{\partial\overline{\lambda}^{j}}((\Phi^{T})^{-1}\mathcal{K}_{T,\lambda^{i}}\Phi^{T})\\
&=&(\Phi^{T})^{-1}\frac{\partial^{j}}{\partial\overline{\lambda}^{j}}\mathcal{K}_{T,\lambda^{i}}\Phi^{T}\\
&=&(\Phi^{T})^{-1}\mathcal{K}_{T,\lambda^{i}\overline{\lambda}^{j}}\Phi^{T},\ i,j\geq1.
\end{array}
\end{equation}
That is, $\Phi^{T}\mathcal{K}_{S,\lambda^{i}\overline{\lambda}^{j}}=\mathcal{K}_{T,\lambda^{i}\overline{\lambda}^{j}}\Phi^{T}$ for all $i,j\geq1$.
\end{proof}
\begin{rem}Based on the observations of Proposition \ref{pro4.1}, the following question is natural.
If the curvatures of the two Cowen-Douglas operators and their covariant partial derivatives satisfy equation (\ref{4.99}), then does it follow that  these two operators are unitarily equivalent?

\end{rem}
\section{The irreducible decompositions of Cowen-Douglas operators}
In this section, we mainly discuss the uniqueness of the direct sum decomposition into irreducible components, modulo a permutation, of an operator in the Cowen-Douglas class.

\begin{defn}
Let $T\in \mathcal{B}_{n}(\Omega)$, $T=(M^{*}_{z}, \mathcal{H}, K)$,
where $\mathcal{H}$ be the analytic function space with reproducing kernel $K$. Suppose that $T$ has the following (IR) decomposition
$$T=\bigoplus\limits_{i=1}^{t}T_{i}^{(m_{i})},\ where \,\,T_{i}\in \mathcal{B}_{n_{i}}(\Omega).$$
Operator $T$ is said to have a unique (IR) decomposition up to unitary equivalence, if for any $\widetilde{T}\sim_{u} T$ with the (IR) decomposition
$$\widetilde{T}=\bigoplus\limits_{i=1}^{s}\widetilde{T}_{i}^{(l_{i})},\ where \,\,\widetilde{T}_{i}\in \mathcal{B}_{h_{i}}(\Omega),$$
then following statements hold:
\begin{itemize}
  \item[(1)] $t=s$;
  \item[(2)] there is a permutation $\pi$ on $\{1,2,\cdots,t\}$,
such that $T_{i}\sim_{u}\widetilde{T}_{\pi(i)}$ and $m_{i}=l_{\pi(i)},\ 1\leq i\leq t$.
\end{itemize}
\label{dingyi}
\end{defn}
In this section, we will prove the following theorem.
\begin{thm}\label{Thm}
Let $T\in\mathcal{B}_{n}(\Omega)$. $T$ has a unique (IR) decomposition up to unitary equivalence.
\end{thm}
Before proving this theorem, we need to do some preliminary work.
\begin{lem}\label{Lyin3}
Let $T,\widetilde{T}\in{\mathcal L}({\mathcal H})$. If $\widetilde{T}=U^{*}TU$ for some unitary $U$, then $\mathcal{A}'(T,T^{*})=U\mathcal{A}'(\widetilde{T},\widetilde{T}^{*})U^*$.
\end{lem}
\begin{proof}
If $\widetilde{T}=U^{*}TU$ for some unitary $U$ and $X\in\mathcal{A}'(T)$, then we obtain the following
$$U^{*}TUU^{*}XU=U^{*}XUU^{*}TU,\ \widetilde{T}U^{*}XU=U^{*}XU\widetilde{T}.$$
Thus, $U^{*}XU\in\mathcal{A}'(\widetilde{T})$. By a similar calculation, we have that if $X\in\mathcal{A}'(T^{*})$, then $U^{*}XU\in\mathcal{A}'(\widetilde{T}^{*})$.
That is to say, $U^{*}\mathcal{A}'(T,T^{*})U\subset\mathcal{A}'(\widetilde{T},\widetilde{T}^{*})$. Similarly, we have $U\mathcal{A}'(\widetilde{T},\widetilde{T}^{*})U^{*}\subset\mathcal{A}'(T,T^{*})$. It follows that $\mathcal{A}'(\widetilde{T},\widetilde{T}^{*})\subset U^{*}\mathcal{A}'(T,T^{*})U$ and
$\mathcal{A}'(T,T^{*})=U\mathcal{A}'(\widetilde{T},\widetilde{T}^{*})U^*$.
\end{proof}
\begin{rem}\label{4.95}
Suppose that $T=\bigoplus_{i=1}^{t}T_{i}^{(m_{i})},\ \widetilde{T}=\bigoplus_{i=1}^{s}\widetilde{T}_{i}^{(l_{i})}\in \mathcal{B}_n(\Omega)$ are (IR) decompositions of $T$ and $\widetilde{T}$, respectively.  
If $\widetilde{T}=U^{*}TU$ for some unitary $U$, according to Lemma \ref{Lyin3}, we have $\mathcal{A}'(T,T^{*})=U\mathcal{A}'(\widetilde{T},\widetilde{T}^{*})U^*$, that is, $\mathcal{A}'(T,T^{*})\cong\mathcal{A}'(\widetilde{T},\widetilde{T}^{*})$.
By Lemma \ref{Lyin2}, we know that $\mathcal{A}'(T, T^{*})\cong \sum\limits_{i=1}^t\oplus \mathcal{M}_{m_i}(\mathbb{C})$ and $\mathcal{A}'(\widetilde{T},\widetilde{T}^{*})\cong \sum\limits_{i=1}^s\oplus \mathcal{M}_{l_i}(\mathbb{C})$.
Thus, $\sum\limits_{i=1}^t\oplus \mathcal{M}_{m_i}(\mathbb{C})\cong \sum\limits_{i=1}^s\oplus \mathcal{M}_{l_i}(\mathbb{C})$. It follows that $t=s$ and there is a permutation $\pi$ on $\{1,2,\cdots,t\}$, such that $m_{i}=l_{\pi(i)}, 1\leq i\leq t$.

Let the (IR) decomposition of $T\bigoplus\widetilde{T}$ be $\bigoplus_{i=1}^{k}\widehat{T}_{i}^{(k_{i})}$ for some positive integer $k\geq t$.
If $T\sim_{u}\widetilde{T}$, then $T^{(2)}\sim_{u}T\bigoplus\widetilde{T}$, that is, $\bigoplus_{i=1}^{t}T_{i}^{(2m_{i})}\sim_{u}\bigoplus_{i=1}^{k}\widehat{T}_{i}^{(k_{i})}$.
Based on the above analysis, we know $t=k$. That means, $\{n_i|T_{i}\in\mathcal{B}_{n_i}(\Omega),1\leq i\leq t\}=\{n'_i|\widetilde{T}_{i}\in\mathcal{B}_{n'_i}(\Omega),1\leq i\leq t\}$.
Thus, there is a permutation $\pi'$ on $\{1,2,\cdots,t\}$, such that the index of $T_{i}, \widetilde{T}_{\pi'(i)}, 1\leq i\leq t$ are the same.
\label{rem}
\end{rem}

By what was discussed above, we will give a complete proof of Theorem \ref{Thm}.

\begin{proof}
Let $T\in\mathcal{B}_{n}(\Omega)$. Suppose that there exists an operator $\widetilde{T}\in\mathcal{B}_{n}(\Omega)$, such that $T\sim_{u} \widetilde{T}$.
By Remark \ref{4.95}, we can write the (IR) decompositions of $T$ and $\widetilde{T}$ as
\begin{equation}\label{e1}
T=\bigoplus_{i=1}^{t}T_{i}^{(m_{i})}=\bigoplus\limits_{j=1}^{k}S_{j}, \, \widetilde{T}=\bigoplus_{i=1}^{t}\widetilde{T}_{i}^{(l_{i})}=\bigoplus\limits_{j=1}^{k}\widetilde{S}_{j},
\end{equation}
where the index of $T_{i}, \widetilde{T}_{i}$ are the same for all $1\leq i\leq t$ and
$$S_{1}=\cdots=S_{m_{1}}=T_{1}, S_{m_{1}+1}=\cdots=S_{m_{1}+m_{2}}=T_{2}, \cdots, S_{m_{1}+\cdots+m_{t-1}+1}=\cdots=S_{m_{1}+\cdots+m_{t}}=T_{t};$$
$$\widetilde{S}_{1}=\cdots=\widetilde{S}_{l_{1}}=\widetilde{T}_{1}, \widetilde{S}_{l_{1}+1}=\cdots=\widetilde{S}_{l_{1}+l_{2}}=\widetilde{T}_{2}, \cdots, \widetilde{S}_{l_{1}+\cdots+l_{t-1}+1}=\cdots=\widetilde{S}_{l_{1}+\cdots+l_{t}}=\widetilde{T}_{t}.$$
Suppose that $S_{i}=(M^{*}_{z}, \mathcal{H}_{i}, K_{i}),\ \widetilde{S}_{i}=(M^{*}_{z}, \widetilde{\mathcal{H}}_{i}, \widetilde{K}_{i}), 1\leq i\leq k$.
Then for all $\mu, \lambda\in\Omega^{*}$, we can express the reproducing kernels of the spaces corresponding to $T$ and $\widetilde{T}$ as
$$K(\mu, \lambda)=diag(K_{1}(\mu, \lambda),\cdots,K_{k}(\mu, \lambda)), \,\widetilde{K}(\mu, \lambda)=diag(\widetilde{K}_{1}(\mu, \lambda),\cdots,\widetilde{K}_{k}(\mu, \lambda)). $$
By Theorem \ref{L1}, we know that if $T\sim_{u}\widetilde{T}$, then there exist an analytic invertible matrix $\Phi=(\!(\Phi_{i,j})\!)_{i,j=1}^{k}$,
such that
\begin{equation}
K(\mu, \lambda)=\Phi(\mu)\widetilde{K}(\mu, \lambda)\overline{\Phi^{T}(\lambda)}
\label{(2.1)}
\end{equation}
for all $\mu, \lambda\in\Lambda$, a open subset of $\Omega^{*}$.
Let $\Psi=\Phi^{-1}=(\!(\Psi_{i,j})\!)_{i,j=1}^{k}.$
Obviously, $\Psi$ is holomorphic. Consequently, formula (\ref{(2.1)}) can be converted into the following two forms:
\begin{equation}
\Psi_{i,j}(\mu)K_{j}(\mu, \lambda)=\widetilde{K}_{i}(\mu, \lambda)\overline{\Phi^{T}_{j,i}(\lambda)},\ \mu, \lambda\in\Lambda,\ 1\leq i,j\leq k;
\label{(2.2)}
\end{equation}
and
\begin{equation}
K_{j}(\mu, \lambda)\overline{\Psi^{T}_{i,j}(\lambda)}=\Phi_{j,i}(\mu)\widetilde{K}_{i}(\mu, \lambda),\ \mu, \lambda\in\Lambda,\ 1\leq i,j\leq k.
\label{(2.3)}
\end{equation}
Combining formulas (\ref{(2.2)}), (\ref{(2.3)}) and multiplying $\Phi_{j,i}(\mu)$, $\Psi_{i,j}(\mu)$ on the left sides of equations (\ref{(2.2)}) and (\ref{(2.3)}) respectively, we obtain
$$\begin{array}{lll}
\Phi_{j,i}(\mu)\Psi_{i,j}(\mu)K_{j}(\mu, \lambda)
&=&\Phi_{j,i}(\mu)\widetilde{K}_{i}(\mu, \lambda)\overline{\Phi_{j,i}(\lambda)}^{T}\\
&=&K_{j}(\mu, \lambda)\overline{\Psi_{i,j}(\lambda)}^{T}\overline{\Phi_{j,i}(\lambda)}^{T}\\
&=&K_{j}(\mu, \lambda)\overline{\Phi_{j,i}(\lambda)\Psi_{i,j}(\lambda)}^{T},\ \mu, \lambda\in\Lambda,\ 1\leq i,j\leq k,
\end{array}$$
and
$$\begin{array}{lll}
\Psi_{i,j}(\mu)\Phi_{j,i}(\mu)\widetilde{K}_{i}(\mu, \lambda)
&=&\Psi_{i,j}(\mu)K_{j}(\mu, \lambda)\overline{\Psi_{i,j}(\lambda)}^{T}\\
&=&\widetilde{K}_{i}(\mu, \lambda)\overline{\Phi_{j,i}(\lambda)}^{T}\overline{\Psi_{i,j}(\lambda)}^{T}\\
&=&\widetilde{K}_{i}(\mu, \lambda)\overline{\Psi_{i,j}(\lambda)\Phi_{j,i}(\lambda)}^{T},\ \mu, \lambda\in\Lambda,\ 1\leq i,j\leq k.
\end{array}$$
Since $S_{i}, \widetilde{S}_{i}\in(IR),\ 1\leq i\leq k$, according to Lemma \ref{Lyin} and the proof of Theorem \ref{KK1},
there exist $c_{i,j}\geq 0$ such that the following holds
\begin{equation}\label{(2.4)}
\Phi_{j,i}\Psi_{i,j}=c_{i,j}\,I,\ \Psi_{i,j}\Phi_{j,i}=c_{i,j}\,I,\ 1\leq i,j\leq k.
\end{equation}
Since $\Phi\Psi=\Psi\Phi=I$, it follows that
\begin{equation}
\Phi(\mu)\Psi(\mu)=\biggl(\sum\limits_{t=1}^{k}\Phi_{i,t}(\mu)\Psi_{t,j}(\mu)\biggl)_{i,j=1}^{k}=I,\  \Psi(\mu)\Phi(\mu)=\biggl(\sum\limits_{r=1}^{k}\Psi_{i,r}(\mu)\Phi_{r,j}(\mu)\biggl)_{i,j=1}^{k}=I.
\label{(2.6)}
\end{equation}
Thus, combining formulas (\ref{(2.4)}) and (\ref{(2.6)}), we obtain
\begin{equation}\label{(2.8)}
\sum\limits_{i=1}^{k}c_{i,j}=1, 1\leq j\leq k;\ \sum\limits_{j=1}^{k}c_{i,j}=1, 1\leq i\leq k.
\end{equation}
Setting $C:=(\!(c_{i,j})\!)_{i,j=1}^{k}$. The equation (\ref{(2.8)}) shows that the sum of elements in each row and each column of $C$ is one.

The size of $\Phi_{i,j}, \Psi_{i,j}$ are determined by $S_{j}$ and $\widetilde{S}_{i}$.
If there exists some $c_{i,j}\neq 0$, that is, $c_{i,j}>0$, from equation (\ref{(2.4)}), we obtain
$(\frac{\Phi_{j,i}(\mu)}{\sqrt{c_{i,j}}})^{-1}=\frac{\Psi_{i,j}(\mu)}{\sqrt{c_{i,j}}},1\leq i,j\leq k$.
And using equation (\ref{(2.2)}), we have that
$K_{j}(\mu, \lambda)=\frac{\Phi_{j,i}(\mu)}{\sqrt{c_{i,j}}}\widetilde{K}_{i}(\mu, \lambda)\frac{\Phi_{j,i}^{*}(\lambda)}{\sqrt{c_{i,j}}}(\mu, \lambda\in\Lambda).$
Therefore, $S_{j}\sim_{u}\widetilde{S}_{i}$, which is obtained from Theorem \ref{L1}.
According to the characteristics of decompositions of $T$ and $\widetilde{T}$, we can divide $C$ into
$$\left(\begin{matrix} C_{1,1} & C_{1,2} & \cdots & C_{1,t}\\
C_{2,1} & C_{2,2} & \cdots & C_{2,t}\\
\vdots & \vdots  & & \vdots\\
C_{t,1} & C_{t,2} & \cdots & C_{t,t}\\
\end{matrix}\right) ,$$
where $C_{i,j}$ has $l_{i}$ rows and $m_{j}$ columns. It is easy to see that if $C_{i,j}\neq0$ means $T_{j}\sim_{u}\widetilde{T}_{i}$.

We claim that the following three statements are true,
\begin{itemize}
    \item[(1)]
    For any $i\in A:=\{1, 2, \cdots, t\}$, in the set $C_{i,-}:=\{C_{i,1}, C_{i,2}, \cdots, C_{i,t}\}$, there exists $j_{0}\in A$, such that $C_{i,j_{0}}\neq 0$ and $C_{i,j}= 0\;(j\neq j_{0})$;
    \item[(2)]
    For any $j\in A$, in the set $C_{-,j}:=\{C_{1,j}, C_{2,j}, \cdots, C_{t,j}\}$, there exists $i_{0}\in A$, such that $C_{i_{0},j}\neq 0$ and $C_{i,j}= 0\;(i\neq i_{0})$;
    \item[(3)]
    If $C_{i,j}\neq 0$, then $C_{i,j}$ is a square matrix.
\end{itemize}

For the matrix $C=(\!(c_{i,j})\!)_{i,j=1}^{k}$, we have already proved equation (\ref{(2.8)}).
Clearly, for any $p,q,\ 1\leq p,q\leq k$, there exist $p_{0},q_{0},\ 1\leq p_{0},q_{0}\leq k$, such that $c_{p, q_{0}}\neq 0,\ c_{p_{0}, q}\neq 0$. That is to say, for any  $i,j\in A$, we have $C_{i,j_{0}}\neq 0,\ C_{i_{0},j}\neq 0$, where $i_{0}, j_{0}\in A$.
Suppose that there are $i_{1}, j_{1}(i_{0}\neq i_{1}, j_{0}\neq j_{1})$, satisfy $C_{i,j_{1}}\neq 0, C_{i_{1},j}\neq 0$.
Then we obtain that $T_{j_{0}}\sim_{u} \widetilde{T}_{i}, \ T_{j_{1}}\sim_{u} \widetilde{T}_{i};\
T_{j}\sim_{u} \widetilde{T}_{i_{0}}, T_{j}\sim_{u} \widetilde{T}_{i_{1}}. $
By Definition \ref{defe}, it is a contradiction. So we complete the proof of (1) and (2).
And there is only one $C_{i, j}$ non-zero in every row and column of matrix $C$. Otherwise, a row or a column of $C$ will be zero. This contradicts (1) or (2).
Next, we will prove that if the matrix $C_{ij}$ is non-zero for any pair of indices $i,j$, then it must
be a square matrix. Suppose that $C_{\widetilde{i},\widetilde{j}}\neq0, \,\,\,\widetilde{i}, \widetilde{j}\in A$.
Adding up all the elements of $C_{\widetilde{i}, \widetilde{j}}$, by equation (\ref{(2.8)}), we obtain that the sum is equal to both $m_{\widetilde{i}}$ and $l_{\widetilde{j}}$,
then $m_{\widetilde{i}}=l_{\widetilde{j}}$, that is, $C_{\widetilde{i},\widetilde{j}}$ is a square matrix. This proves assertion (3).

For every $T_{i}$, there is a unique $\widetilde{T}_{j}$ such that $T_{i}\sim_{u}\widetilde{T}_{j},1\leq i,j\leq t$. In other words, there is a permutation $\pi$ on $\{1, 2, \cdots t\}$,
such that $T_{i}\sim_{u}\widetilde{T}_{\pi(i)}$ and $m_{i}=l_{\pi(i)}, i\in A$.
By Definition \ref{dingyi}, $T$ has a unique decomposition up to unitary equivalence.
\end{proof}

\begin{rem}
Theorem \ref{Thm} can also be proved using the notion of a normalized kernel.

Let $T\in \mathcal{B}_{n}(\Omega)$ and $T=(M_z^*,\mathcal{H}_K,K)$. Suppose that the (IR) decomposition of $T$ is $T_{1}\oplus T_{2}$, $T_i\in \mathcal{B}_{n_i}(\Omega),i=1,2$.
Here we illustrate this method by proving that $T$ has a unique (IR) decomposition up to unitary equivalence.
Suppose that $\widetilde{T}\sim_{u}T$ and $\widetilde{T}=(M_z^*,\mathcal{H}_{\widetilde{K}},\widetilde{K})$. By Remark \ref{4.95}, we know that the (IR) decomposition of $\widetilde{T}$ must be
$\widetilde{T}=\widetilde{T}_{1}\oplus \widetilde{T}_{2}$, $i=1,2$.
It is well known that every reproducing kernel on analytic Hilbert space can be normalized. We think that $K,\widetilde{K}$ have been normalized at $\lambda_{0}\in \Omega^{*}$.
If $\widetilde{T}$ and $T$ are unitarily equivalent via some intertwining unitary map $\Gamma:\mathcal H_K \to \mathcal H_{\widetilde{K}}$, then the map $\Gamma$ must be of the form $(\Gamma f) (z) = U (f(z))$, that is, $K(\mu,\lambda)=U^{*}\widetilde{K}(\mu,\lambda)U,\mu,\lambda\in\Omega^*$.
Since $K, \widetilde{K}$ are real (positive) analytic, and admits power series expansion around $\lambda_{1}\in\Omega$. Assume $\lambda_{1}=0$ for simplicity. We have that $$K(\mu,\lambda)=\sum_{m,n=0}^{\infty}A_{mn}\mu^{m}\overline{\lambda}^{n},\ \widetilde{K}(\mu,\lambda)=\sum_{m,n=0}^{\infty}B_{mn}\mu^{m}\overline{\lambda}^{n}.$$
The direct decomposition of $T$ and $\widetilde{T}$ means that the form of their reproducing kernels are $K=K_{1}\oplus K_{2},\widetilde{K}=\widetilde{K}_{1}\oplus \widetilde{K}_{2}$,
where $K_{i}, \widetilde{K}_{i}$ are normalized kernels corresponding to $T_{i},\ \widetilde{T}_{i}\ (i=1, 2)$ at $\lambda_{0}\in \Omega^{*}$.
Thus, all the matrix coefficients $A_{mn},B_{mn},m,n\geq0$ in the power series expansions must be simultaneously decomposable.
It follows that $A_{mn}^{1}\oplus A_{mn}^{2}=U^{*}(B_{mn}^{1}\oplus B_{mn}^{2})U,\ m, n\geq0,$ where
$$K_{i}(\mu,\lambda)=\sum_{m,n=0}^{\infty}A^{i}_{mn}\mu^{m}\overline{\lambda}^{n},\ \widetilde{K}_{i}(\mu,\lambda)=\sum_{m,n=0}^{\infty}B^{i}_{mn}\mu^{m}\overline{\lambda}^{n},i=1,2.$$

If possible, first assume that all the $A^{(1)}_{mn} = A^{(2)}_{mn}$ for all $m,n$. We must also have that $B^{(1)}_{mn} = B^{(2)}_{mn}$ for all $m,n$.  Since these two operators were assumed to have a decomposition with the components in the class $(IR)$, it follows that $\{A_{m,n}^{(1)}: m,n \in \mathbb N\}$ is a set of irreducible operators. Similarly,
$\{B_{m,n}^{(1)}: m,n \in \mathbb N\}$ is a set of irreducible operators.  These two sets of operators cannot be unitarily equivalent by a fixed unitary unless the unitary is of the form: $\alpha I,\,\alpha \in \mathbb T$.	So, in this case, $T=\tilde{T}$.

Now, let us assume without loss of generality, $A^{(1)}_{m,n} \not = A_{mn}^{(2)}$ for some fixed pair $m,n$. Then for this pair $m,n$ we have that $A_{mn}$ equals the direct sum $D^{(1)}_{mn} \oplus  D^{(2)}_{mn}$ of diagonal matrices since $A_{mn}$ is positive definite. The same is true of $B_{mn}$. If the operator $T$ and $\tilde{T}$ are of rank 2, then we see that our hypothesis implies the existence of a unitary intertwining two diagonal matrices of size $2$ with distinct eigenvalues. This is possible only if U is a scalar times identity or is a permutation.

If the rank of the operator $T$ is greater than 2, we consider the algebra generated by all the the coefficients of the kernel $K$ and the algebra generated by all the the coefficients of $\tilde{K}$ is not irreducible but is the direct sum of two irreducible algebras.  This is because the algebra is generated by the set $\{A^{(1)}_{mn} \oplus A^{(2)}_{mn}: m,n \in \mathbb N\}$
and the coefficients $A_{mn}^{i}, B_{mn}^{i},m,n\geq0$ must not be simultaneously reducible for each $i=1,2$. Again, in this case, the intertwining unitary operator must intertwine these two algebras and hence must be of the form $c_1 I \oplus c_2 I$ or a permutation $\left(\begin{smallmatrix} 0 & c_1 I\\ c_2 I & 0 \end{smallmatrix}\right)$, for some choice of unimodular scalars $c_1, c_2$.

Based on the above analysis, we know that for each $i=1,2$, there is a permutation $\pi$ on $\{1, 2\}$,
such that
$A_{mn}^{i}\sim_{u}B_{mn}^{\pi(i)}$ for all $m, n\geq0$ and $T_{i}\sim_{u}\widetilde{T}_{\pi(i)}$. Thus, $T$ has a unique (IR) decomposition up to unitary equivalence.
\end{rem}
\section{Application }
In this section, we will discuss the structure of the intertwining matrix in equation (\ref{(4.8990)}) by using Theorem \ref{KK1} and Theorem \ref{Thm}.

Suppose that $T=T_1\oplus T_2$ and $\widetilde{T}=\widetilde{T}_1\oplus \widetilde{T}_2$, where $T_i=(M^{*}_{z}, \mathcal{H}_{i}, K_{i}), \widetilde{T}_i=(M^{*}_{z}, \widetilde{\mathcal{H}}_{i}, \widetilde{K}_{i})\in {\mathcal B}_1(\Omega),i=1,2$ and $T_1\not\sim_{u}T_2$, $\widetilde{T}_1\not\sim_{u}\widetilde{T}_2$. Suppose that $T$ and $\widetilde{T}$ are unitarily equivalent. By Theorem \ref{L1}, there exists an invertible holomorphic matrix valued function $\Phi=((\phi_{i,j}))_{2\times 2}$ such that \begin{equation}
\left(\begin{matrix} K_{1}(\mu, \lambda) & 0 \\
0 & K_{2}(\mu, \lambda)\\
\end{matrix}\right)=\left(\begin{matrix} \phi_{1,1}(\mu) & \phi_{1,2}(\mu) \\
\phi_{2,1}(\mu) & \phi_{2,2}(\mu)\\
\end{matrix}\right)\left(\begin{matrix} \widetilde{K}_{1}(\mu, \lambda) & 0 \\
0 & \widetilde{K}_{2}(\mu, \lambda)\\
\end{matrix}\right)\left(\begin{matrix} \bar{\phi}_{1,1}(\lambda) & \bar{\phi}_{2,1}(\lambda) \\
\bar{\phi}_{1,2}(\lambda) & \bar{\phi}_{2,2}(\lambda)\\
\end{matrix}\right)
\end{equation}
Furthermore, by Theorem \ref{KK1}, we have that $$\frac{\partial}{\partial\lambda}[K_{\widetilde{T}}^{-1}(\lambda, \lambda)\Phi(\lambda)^{-1}K_{T}(\lambda, \lambda)]=0$$
for all $\lambda\in \Lambda$, a open subset of $\Omega^{*}.$
Set $\Phi^{-1}=\Psi=((\psi_{i,j}))_{2\times 2}$, then we have that
$$\left(\begin{matrix} \widetilde{K}^{-1}_{1}(\lambda, \lambda)\psi_{1,1}(\lambda)K_1(\lambda, \lambda) &\ \widetilde{K}^{-1}_{1}(\lambda, \lambda)\psi_{1,2}(\lambda)K_2(\lambda, \lambda) \\
\widetilde{K}^{-1}_{2}(\lambda, \lambda)\psi_{2,1}(\lambda)K_1(\lambda, \lambda)&\ \widetilde{K}^{-1}_{2}(\lambda, \lambda)\psi_{2,2}(\lambda)K_2(\lambda, \lambda)\\
\end{matrix}\right)$$
is anti-holomorphic. Suppose that $\psi_{1,1}$ is nonzero function. By Theorem \ref{KK1} in the case of index is one, we can see $T_1\sim_{u}\widetilde{T}_1$. If $\psi_{1,2}$ is also nonzero function, then $T_2\sim_{u}\widetilde{T}_1$. Which is a contradiction to $T_1\not\sim_{u}T_2$.
Thus, since $\Phi,\Psi$ are invertible, we have that $\Phi$ only have the following two forms:
$$\Phi=\left(\begin{matrix} \phi_{1,1}&0 \\
0 & \phi_{2,2}\\
\end{matrix}\right)~~\mbox{or}~~\Phi=\left(\begin{matrix} 0&\phi_{1,2} \\
\phi_{2,1} & 0\\
\end{matrix}\right).$$

For the general case, set $T=\bigoplus_{i=1}^{t}T_{i}^{(m_{i})}\in \mathcal{B}_{n}(\Omega)$ satisfies Definition \ref{defe}.
If $\widetilde{T}\sim_{u}T$, by Theorem \ref{Thm}, we know $\widetilde{T}=\bigoplus_{i=1}^{t}\widetilde{T}_{i}^{(l_{i})}$ and there exists a permutation $\pi$ on $\{1, 2, \cdots t\}$ such that $T_{i}\sim_{u}\widetilde{T}_{\pi(i)},\ m_{i}=l_{\pi(i)}$.
Setting $k:=\sum_{i=1}^{t}m_{i}=\sum_{i=1}^{t}l_{i}$. Next, we write $T$ and $\widetilde{T}$ in terms of equation (\ref{e1}).
Suppose that $S_{i}=(M^{*}_{z}, \mathcal{H}_{i}, K_{i}),\ \widetilde{S}_{i}=(M^{*}_{z}, \widetilde{\mathcal{H}}_{i}, \widetilde{K}_{i}),\ 1\leq i\leq k$.

According to Theorem \ref{L1}, there exists a holomorphic invertible matrix $\Phi=(\!(\Phi_{i,j})\!)_{i,j=1}^{k}$, such that
$$\begin{array}{lll}
K(\mu, \lambda)
&=&\Phi(\mu)\widetilde{K}(\mu, \lambda)\Phi^{*}(\lambda)\\
&=&\left (\begin{smallmatrix}
\sum\limits_{i=1}^{k}\Phi_{1,i}(\mu)\widetilde{K}_{i}(\mu, \lambda)\Phi_{1,i}^{*}(\lambda) & * & \cdots & *\\
* & \sum\limits_{i=1}^{k}\Phi_{2,i}(\mu)\widetilde{K}_{i}(\mu, \lambda)\Phi_{2,i}^{*}(\lambda) & \cdots & *\\
\vdots & \vdots & \ddots & \vdots\\
* & * & \cdots & \sum\limits_{i=1}^{k}\Phi_{k,i}(\mu)\widetilde{K}_{i}(\mu, \lambda)\Phi_{k,i}^{*}(\lambda) \\
\end{smallmatrix}\right),
\end{array}$$
that is,
\begin{equation}\label{4.161}
K_{j}(\mu, \lambda)=\sum_{i=1}^{k}\Phi_{j,i}(\mu)\widetilde{K}_{i}(\mu, \lambda)\Phi_{j,i}^{*}(\lambda),\ 1\leq j\leq k,
\end{equation}
for all $\mu, \lambda\in\Lambda$, a open subset of $\Omega^{*}.$ Setting $\Phi^{-1}=\Psi=(\!(\Psi_{i,j})\!)_{i,j=1}^{k}$.
By Lemma \ref{Lyin} and the proof of Theorem \ref{Thm}, we have
\begin{equation}\label{4.162}
\Psi_{i,j}(\mu)K_{j}(\mu, \lambda)=\widetilde{K}_{i}(\mu, \lambda)\Phi_{j,i}^{*}(\lambda),\ K_{j}(\mu, \lambda)\Psi^{*}_{i,j}(\lambda)=\Phi_{j,i}(\mu)\widetilde{K}_{i}(\mu, \lambda),\ \mu, \lambda\in\Lambda,
\end{equation}
\begin{equation}\label{4.171}
\Psi_{i,j}\Phi_{j,i}=c_{i,j}I,\,\Phi_{j,i}\Psi_{i,j}=c_{i,j}I,\,c_{i,j}\geq 0,\ 1\leq i,j\leq k,
\end{equation}
and
\begin{equation}\label{4.163}
\sum\limits_{i=1}^{k}c_{i,j}=1, 1\leq j\leq k,\ \sum\limits_{j=1}^{k}c_{i,j}=1,\ 1\leq i\leq k.
\end{equation}
By equations (\ref{4.162}) and (\ref{4.171}), we can imply that
\begin{equation}\label{4.172}
c_{i,j}K_{j}(\mu, \lambda)=\Phi_{j,i}(\mu)\widetilde{K}_{i}(\mu, \lambda)\Phi_{j,i}^{*}(\lambda),\ 1\leq j\leq k.
\end{equation}
Let $C=(\!(c_{i,j})\!)_{i,j=1}^{k}$. If $c_{i,j}>0$, it follows that $\Psi_{i,j},\ \Phi_{j,i}$ are invertible and $S_{j}\sim_{u}\widetilde{S}_{i}$.
Similar to Theorem \ref{Thm}, we divide $C$ into $(\!(C_{i,j})\!)_{i,j=1}^{t}$, where $C_{i,j}$ has $l_{i}$ rows and $m_{j}$ columns.
For any $j,\, 1\leq j\leq t$, we obtain $C_{\pi(j),j}\neq 0$ and $C_{i,j}=0$ for all $i\neq \pi(j)$.

Suppose that $C_{\pi(p),p}\neq 0,\ 1\leq p\leq t$. Let
$$\widetilde{i}:=
\left\{
\begin{array}{cc}
0,\ \pi(p)=1;\\
\sum\limits_{r=1}^{\pi(p)-1}l_{r},\ \pi(p)> 1,
\end{array}
\right. \widetilde{j}:=
\left\{
\begin{array}{cc}
0,\ p=1;\\
\sum\limits_{r=1}^{p-1}m_{r},\ p>1.
\end{array}
\right.$$
Then $C_{\pi(p),p}=(\!(c_{i,j})\!)_{l_{\pi(p)}\times m_{p}}$, where $\widetilde{i}+1\leq i\leq \widetilde{i}+l_{\pi(p)},\ \widetilde{j}+1\leq j\leq\widetilde{j}+m_{p}$.
Thus, from equation (\ref{4.172}), we have
$$K_{j}(\mu, \lambda)=\sum_{i=\widetilde{i}+1}^{\widetilde{i}+l_{\pi(p)}}c_{ij}K_{j}(\mu, \lambda)=\sum_{i=\widetilde{i}+1}^{\widetilde{i}+l_{\pi(p)}}\Phi_{j,i}(\mu)\widetilde{K}_{i}(\mu, \lambda)\Phi_{j,i}^{*}(\lambda),\ 1\leq j\leq k.$$
By using equations (\ref{4.161}),(\ref{4.163}) and (\ref{4.172}),
$$\begin{array}{lll}
 & &\sum_{i=1}^{\widetilde{i}}c_{ij}K_{j}(\mu, \lambda)+\sum_{i=\widetilde{i}+l_{\pi(p)}+1}^{k}c_{ij}K_{j}(\mu, \lambda)\\
&=&\sum_{i=1}^{\widetilde{i}}\Phi_{j,i}(\mu)\widetilde{K}_{i}(\mu, \lambda)\Phi_{j,i}^{*}(\lambda)+\sum_{i=\widetilde{i}+l_{\pi(p)}+1}^{k}\Phi_{j,i}(\mu)\widetilde{K}_{i}(\mu, \lambda)\Phi_{j,i}^{*}(\lambda)\\
&=&0,\ 1\leq j\leq k.
\end{array}$$
Note that $\widetilde{K}_{i}(\mu, \lambda)$ is positive and $\Phi_{j,i}(\mu)\widetilde{K}_{i}(\mu, \lambda)\Phi_{j,i}^{*}(\lambda)$ is nonnegative definite, then the last equation means that
$\Phi_{j,i}=0$, when $1\leq i\leq\widetilde{i}$ or $\widetilde{i}+l_{\pi(p)}+1\leq i\leq k$, $1\leq j\leq k$.
We divide $\Phi$ into $(\!(\widetilde{\Phi}_{i,j})\!)_{i,j=1}^{t}$, where $\widetilde{\Phi}_{i,j}$ has $m_{i}$ rows and $l_{j}$ columns.
According to the discussion above, we obtain $\widetilde{\Phi}_{i,\pi(i)}\neq 0$ and $\widetilde{\Phi}_{i,j}= 0(j\neq\pi(i))$.
Hence, $\Phi_{i,j}$ in $\Phi$ is either  invertible or zero.

\textbf{Acknowledgements:}

The authors thank the referee for carefully reading the paper and pointing out some
helpful comments and suggestions. The first author would like to thank Dr. Soumitra Ghara and Li Chen for many valuable discussions.
The first author was supported by National Natural Science Foundation of China (Grant No. 11831006 and
11922108).

\end{document}